\documentclass[12pt,a4paper,twoside,reqno]{amsart}
\title[Harmonic metrics, Hull-Strominger system, stability]{Harmonic metrics for the Hull-Strominger system and stability}

\author[M. Garcia-Fernandez]{Mario Garcia-Fernandez}
\address{Universidad Aut\'onoma de Madrid and Instituto de Ciencias Matem\'aticas (CSIC-UAM-UC3M-UCM)\\ Ciudad
	Universitaria de Cantoblanco\\ 28049 Madrid, Spain}
\email{mario.garcia@icmat.es}

\author[R. Gonzalez Molina]{Raul Gonzalez Molina}
\address{Instituto de Ciencias Matem\'aticas (CSIC-UAM-UC3M-UCM)\\ Nicol\'as Cabrera 13--15, Cantoblanco\\ 28049 Madrid, Spain}
  \email{raul.gonzalez@icmat.es}

\thanks{This work is partially supported by the Spanish Ministry of Science and Innovation, under grants PID2019-109339GA-C32, EUR2020-112265, and I+D+i SEV2015-05554-18-1 funded by MCIN/AEI/10.13039/501100011033}

\usepackage{hyperref,url}
\usepackage{amssymb,latexsym}
\usepackage{amsmath,amsthm}
\usepackage[latin1]{inputenc}
\usepackage{enumerate}
\usepackage{graphicx}
\usepackage[dvipsnames]{xcolor}
\usepackage[all]{xy} \CompileMatrices
\SelectTips{cm}{12} 
\theoremstyle{plain}
\newtheorem{theorem}{Theorem}[section]
\newtheorem{lemma}[theorem]{Lemma}
\newtheorem{corollary}[theorem]{Corollary}
\newtheorem{proposition}[theorem]{Proposition}

\newtheorem{question}[theorem]{Question}
\theoremstyle{definition}
\newtheorem{definition}[theorem]{Definition}
\newtheorem{definition-theorem}[theorem]{Definition-Theorem}

\theoremstyle{remark}
\newtheorem{remark}[theorem]{Remark}

\numberwithin{equation}{section} 

\newcommand{\tr}{\operatorname{tr}}
\newcommand{\Id}{\operatorname{Id}}

\newcommand{\End}{\operatorname{End}}
\newcommand{\Hom}{\operatorname{Hom}}

\newcommand{\ad}{\operatorname{ad}}

\newcommand{\dbar}{\bar{\partial}}

\newcommand{\CC}{{\mathbb C}}

\newcommand{\RR}{{\mathbb R}}

\newcommand{\rk}{\operatorname{rk}}

\newcommand{\surj}{\to\kern-1.8ex\to}

\newcommand{\cA}{\mathcal{A}}

\newcommand{\cG}{\mathcal{G}}
\newcommand{\cL}{\mathcal{L}}

\newcommand{\Lie}{\operatorname{Lie}}

\newcommand{\IP}[1]{\left<#1\right>}

\begin{document}

\maketitle

\begin{center}
{\small \emph{Dedicated to Oscar Garc\'ia-Prada on the occasion of his sixtieth birthday.}}
\end{center}

\begin{abstract}%
 
We investigate stability conditions related to the existence of solutions of the Hull-Strominger system with prescribed balanced class. We build on recent work by the authors, where the Hull-Strominger system is recasted using 
non-Hermitian Yang-Mills connections and holomorphic Courant algebroids. Our main development is a notion of \emph{harmonic metric} for the Hull-Strominger system, motivated by an infinite-dimensional hyperK\"ahler moment map and related to a numerical stability condition, which we expect to exist generically for families of solutions. We illustrate our theory with an infinite number of continuous families of examples on the Iwasawa manifold.

\end{abstract}


\section{Introduction}\label{sec:intro}

The Hull-Strominger system of partial differential equations \cite{HullTurin,Strom} was proposed in \cite{FuYau,LiYau} as a geometrization tool for understanding the moduli space of Calabi-Yau threefolds with topology change (also known as `Reid's Fantasy' \cite{Reid}). The basic geometric data underlying a solution is given by a Calabi-Yau manifold $(X,\Omega)$, possibly non-K\"ahler, and a holomorphic vector bundle $V$ over it satisfying a suitable topological constraint. The degrees of freedom of the system are a conformally balanced hermitian form $\omega$ on $X$, a hermitian metric on $V$, and a connection $\nabla$ on the smooth complex tangent bundle $T^{1,0}$ of $X$. The interest in this geometrization programme has motivated the study of conifold transitions and balanced metrics on non-K\"ahler manifolds \cite{CoPiYau,CoPiGuYau,FuLiYau,Phong}. Nonetheless, recent work by the authors in \cite{GFGM} shows that the initial proposal by Fu, Li, and Yau may possibly need to be refined, due to the existence of obstructions which go beyond the balanced property of $X$ and the Mumford-Takemoto slope stability of $V$.

These new obstructions, which take the form of \emph{Futaki invariants}, are inspired by the recent moment map constructions in \cite{Waldram,GaRuTi3}. The picture suggested by \cite{GaRuTi3} relates the moduli space of solutions of the Hull-Strominger system with a moduli space of \emph{holomorphic string algebroids}, a special class of holomorphic Courant algebroids given by holomorphic extensions 
\begin{equation*}
\begin{split}
0 \longrightarrow T^*_{1,0} \overset{\pi^*}{\longrightarrow} \mathcal{Q} \longrightarrow A \longrightarrow 0,
\end{split}
\end{equation*}
for suitable Atiyah Lie algebroids $A$, endowed with a holomorphic pairing $\IP{,}$ and a bracket $[,]$ on its sheaf of sections. Based on the general principles of the Donaldson-Uhlenbeck-Yau Theorem \cite{Don,UYau}, it is natural to expect that the holomorphic objects $(\mathcal{Q},\IP{,},[,])$ which are parametrized by this moduli space satisfy a stability condition in the sense Geometric Invariant Theory.

Motivated by this picture, in this work we investigate stability conditions related to the existence of solutions of the Hull-Strominger system with prescribed balanced class of the solution in the Bott-Chern cohomology of $X$
$$
\mathfrak{b} = [\|\Omega\|_\omega \omega^{n-1}] \in H^{n-1,n-1}_{BC}(X,\RR).
$$
Even though the system was originally considered in complex dimension three, we shall work here in arbitrary complex dimension $n$ (see Section \ref{sec:HS-HE}). We start by studying the relation between the existence of non-K\"ahler solutions and the Mumford-Takemoto slope polystability of the holomorphic orthogonal bundle $(\mathcal{Q},\IP{,})$ with respect to $\mathfrak{b}$, showing in Proposition \ref{prop:HSnogo} that these two conditions are not compatible in general. In particular, we recover a no-go result for solutions `without the $\nabla$ connection', which goes back to the seminal work of Candelas, Horowitz, Strominger, and Witten \cite{CHSW}. From our point of view, this is a consequence of the slope stability of $(\mathcal{Q},\IP{,})$ with respect to the balanced class of the solution $\mathfrak{b}$, combined with the existence of a holomorphic volume form.

We go on to propose a refined stability condition based on hyperK\"ahler moment maps (see Section \ref{sec:HarmonicMet}), which is inspired by the theory of Higgs bundles for orthogonal groups \cite{GP,Hitchin1987}. We build on a result in \cite{GFGM}, which proves that a solution of the Hull-Strominger system induces an indefinite Hermitian-Einstein metric $\mathbf{G}$ on $(\mathcal{Q},\IP{,})$, that is, satisfying
\begin{equation}\label{eq:HEintro}
F_\mathbf{G} \wedge \omega^{n-1} = 0.
\end{equation}
In particular, the Chern connection $D^\mathbf{G}$ of $\mathbf{G}$ is a \emph{non-Hermitian Yang-Mills connection} in the sense of Kaledin and Verbitsky \cite{KaledinVerbitsky}. This result shall be compared with a physical result in \cite{OssaLarforsSvanes}, which states the Hull-Strominger system is equivalent to a suitable Hermitian-Yang-Mills equation on a Courant algebroid to all orders in perturbation theory. Then, the basic idea is that a solution of the Hull-Strominger system should carry a positive definite \emph{harmonic metric} $\mathbf{H}$ for $(\mathcal{Q},\IP{,},D^\mathbf{G})$, that is, satisfying 
\begin{equation*}
(\nabla^{\mathbf{H}})^* \Psi + i_{ \theta_\omega^\sharp}\Psi = 0,
\end{equation*}
where $D^\mathbf{G} = \nabla^\mathbf{H} + \Psi$ is the unique decomposition of the Chern connection into an $\mathbf{H}$-unitary connection and a \emph{Higgs field}. Here, $\theta_\omega^\sharp$ denotes the dual vector field of the Lee form of $\omega$. Using a different decomposition of $D^\mathbf{G}$ \emph{\`a la Hitchin} \cite{Hitchin1987}, in our main result, Theorem \ref{thm:HSstab}, we prove that the existence of a harmonic metric for the Hull-Strominger system implies a numerical stability condition reminiscent of GIT. 

Our proposal is illustrated with an infinite number of continuous families of solutions on the Iwasawa manifold in Section \ref{sec:examples}. It is interesting to observe that, in our examples, the curvature of the connection which formally plays the role of the connection $\nabla$ in the tangent bundle corresponds precisely to the Higgs field $\Psi$ (see Remark \ref{rem:HSabstract} and Lemma \ref{lemma:Kmmap}). As in \cite{GFGM}, we will assume throughout that the connection $\nabla$ in the tangent bundle is Hermitian-Yang-Mills. This suggests a new point of view on the spurious degrees of freedom corresponding to the connection $\nabla$ in the physical interpretation of the solutions, which have motivated an extensive discussion in the string theory literature (see \cite{McOristSvanes} and references therein).

An open question which we have not been able to solve is to establish a more clear relation between the Dorfman bracket $[,]$ on $\mathcal{Q}$ and the orthogonal connection $D^\mathbf{G}$, in a way that our new stability condition is formulated more naturally in terms of the triple $(\mathcal{Q},\IP{,},[,])$. Even though our picture is mostly conjectural, we expect that this stability condition, along with the notion of harmonic metric that we introduce, will lead to new obstructions to the existence of solutions in future studies, similarly as in \cite{GaJoSt}.

The main results and ideas of this work were presented in September 2022, during the conference `Moduli spaces and geometric structures' at the ICMAT in Madrid, in honour of Oscar Garc\'ia-Prada. With the aim of minimizing the elapsed time between the present arXiv submission and that of \cite{GFGM}, we decided to postpone the former. Very recently, an interesting paper by Pan, Shen, and Zhang appeared \cite{PSZ}, which characterizes the existence of harmonic metrics for non-Hermitian Yang-Mills connections on K\"ahler manifolds. In particular, they provide an alternative proof of our Proposition \ref{prop:harmHKstab} in the K\"ahler case. In the light of the present paper, it would be interesting to explore further applications of the main results in \cite{PSZ} to the Hull-Strominger system.

\section{The Hull-Strominger system and slope stability}\label{sec:HSslope}

\subsection{The Hull-Strominger system and Hermitian-Einstein metrics}\label{sec:HS-HE}

We start by recalling some background on the Hull-Strominger system, including the recent development in \cite{GFGM} relating solutions to these equations to indefinite Hermitian-Einstein metrics on holomorphic Courant algebroids. We will use an abstract definition of the Hull-Strominger system, as considered in \cite[Definition 2.4]{GaRuShTi}, which is valid in arbitrary dimensions. Our construction requires that the connection $\nabla$ in the tangent bundle in the original formulation of the system is Hermitian-Yang-Mills, and hence in our discussion we will implicitly assume this condition (see Remark \ref{rem:HSabstract}).

Let $X$ be a compact complex manifold of dimension $n$ endowed with a holomorphic volume form $\Omega$. Let $V_0$ and $V_1$ denote holomorphic vector bundles over $X$ satisfying
\begin{equation}\label{eq:c1c2abstract}
 ch_2(V_0) = ch_2(V_1) \in H^{2,2}_{BC}(X,\RR).
\end{equation}
Here $H^{p,q}_{BC}(X)$ are the Bott-Chern cohomology groups of the complex manifold $X$, defined by
\begin{equation}\label{eq:BC}
H^{p,q}_{BC}(X)=\frac{\mathrm{ker} \hspace{1mm} d: \Omega^{p,q}(X,\mathbb{C}) \longrightarrow \Omega^{p+q+1}(X,\mathbb{C})}{\mathrm{Im} \hspace{1mm}  dd^{c}: \Omega^{p-1,q-1}(X,\mathbb{C}) \longrightarrow \Omega^{p,q}(X,\mathbb{C})}
\end{equation}
and $H^{p,p}_{BC}(X,\mathbb{R}) \subset H^{p,p}_{BC}(X)$ is the canonical real structure.

\begin{definition}\label{def:HSabstract}
We say that a triple $(g,h_0,h_1)$, where $g$ is a Hermitian metric on $X$ and $h_j$ is a Hermitian metric on $V_j$, $j = 0,1$, satisfies the \emph{Hull-Strominger system} with coupling constant $\alpha \in \RR$ if
\begin{equation}\label{eq:HSabstract}
\begin{split}
F_{h_0}\wedge \omega^{n-1} & =0,\\
F_{h_1}\wedge \omega^{n-1} & =0,\\
d(\|\Omega\|_\omega \omega^{n-1}) & = 0,\\
dd^c \omega - \alpha \tr F_{h_0} \wedge F_{h_0} + \alpha \tr F_{h_1} \wedge F_{h_1} & = 0,
\end{split}
\end{equation}
where $\omega = g(J,)$ is the Hermitian form of $g$.
\end{definition}

\begin{remark}\label{rem:HSabstract}
In the original definition of the Hull-Strominger system in the physics literature \cite{HullTurin,Strom}, one has that $n=3$ and that $V_0$ is isomorphic to the holomorphic tangent bundle $T^{1,0}$ as a smooth complex vector bundle (see \cite{GFGM} for further clarifications on the ansatz considered here). Formally, in our setup, the Chern connection of $h_0$ on $V_0$ plays the role of the connection $\nabla$ on $T^{1,0}$, which is assumed to satisfy the Hermitian-Yang-Mills equations (other choices are possible, as e.g. in \cite{LiYau}). 
\end{remark}

Assume that $(X,\Omega,V_0,V_1)$ admits a solution $(g,h_0,h_1)$ of the Hull-Strominger system \eqref{eq:HSabstract} with coupling constant $\alpha$. As observed in \cite{GFGM}, these data determines a holomorphic orthogonal bundle endowed with an indefinite Hermitian-Einstein metric. Denote by $P$ the holomorphic principal $G$-bundle of split frames of $V_0 \oplus V_1$. Note that $\ad P \cong \End V_0 \oplus \End V_1$ and consider the pairing $\IP{,}: \ad P \otimes \ad P \to \mathbb{C}$ induced by 
\begin{equation}\label{eq:pairingHS}
\IP{,} : = - \alpha \tr_{V_0} + \alpha \tr_{V_1}.
\end{equation}
As a smooth complex vector bundle, the orthogonal bundle determined by the solution is given by
\begin{equation*}\label{eq:Qexpb}
\mathcal{Q} = T^{1,0} \oplus\ad P\oplus T_{1,0}^*,
\end{equation*}
where $T^{1,0}$ and $T_{1,0}^*$ denote respectively the holomorphic tangent and cotangent bundles of $X$, and has non-degenerate symmetric bilinear form
\begin{equation}\label{eq:pairing}
\IP{V + r + \xi , V + r + \xi}  = \xi(V) + \IP{r,r}.
\end{equation}
The holomorphic structure is then given by the Dolbeault operator
\begin{equation}\label{eq:DolQ}
\dbar^{\mathcal{Q}} (V + r + \xi)  = \dbar V + i_V F_\theta^{1,1} + \dbar^\theta r + \dbar \xi - i_{V}(2i\partial \omega) + 2\IP{ F_\theta^{1,1}, r},
\end{equation}
where $\theta$ denotes the Chern connection of the metric $h = h_0 \oplus h_1$ on $V_0 \oplus V_1$. 

\begin{lemma}[\cite{GaRuShTi,GaRuTi2}]\label{lem:HSQ}
The bundle $\mathcal{Q} = T^{1,0} \oplus\ad P\oplus T_{1,0}^*$ endowed with the Dolbeault operator \eqref{eq:DolQ} defines a holomorphic extension of the holomorphic Atiyah algebroid $A_P$ of $P$ by the holomorphic cotangent bundle
\begin{equation}\label{eq:holCoustr}
0 \longrightarrow T^*_{1,0} \overset{\pi^*}{\longrightarrow} \mathcal{Q} \longrightarrow A_P \longrightarrow 0.
\end{equation}
Furthermore, the pairing \eqref{eq:pairing} is non-degenerate and holomorphic.
\end{lemma}

\begin{remark}\label{rem:bracket}
The holomorphic orthogonal bundle $(\mathcal{Q},\IP{,})$ can be endowed with a bracket on holomorphic sections satisfying the Jacobi identity, given by
\begin{equation*}
	\begin{split}
	[V+ r + \xi,W + t + \eta]   = {} & [V,W] - F^{2,0}_\theta(V,W) + \partial^\theta_V t - \partial^\theta_W r - [r,t]\\
	& {} + i_V \partial \eta + \partial (\eta(V)) - i_W\partial \xi\\
	& {} + 2\IP{\partial^\theta r, t} + 2\IP{i_V F_\theta^{2,0}, t} - 2\IP{i_W F_\theta^{2,0}, r}.
\end{split}
\end{equation*}
Combined with the \emph{anchor map} $\pi(V+ r + \xi) = V$ onto $T^{1,0}$, this defines a holomorphic Courant algebroid of string type (see \cite[Proposition 2.4]{GaRuTi2}). We will not use this enhanced structure in this work (cf. Remark \ref{rem:connectionvsbracket}).
\end{remark}

The next result from \cite{GFGM} establishes the relation between solutions of the Hull-Strominger system and indefinite Hermitian-Einstein metrics. Given a solution of the Hull-Strominger system, there is a natural bundle isomorphism 
\begin{equation}\label{eq:Bismutiso}
T^{1,0} \oplus T^*_{1,0} \cong TX \otimes \CC \colon V + \xi \mapsto V - \tfrac{1}{2} g^{-1}\xi
\end{equation}
induced by the metric $g$ of the solution.

\begin{proposition}[\cite{GFGM}]\label{prop:HS-HE}
Assume that $(X,\Omega,V_0,V_1)$ admits a solution $(g,h_0,h_1)$ of the Hull-Strominger system \eqref{eq:HSabstract} with coupling constant $\alpha$. Let $(\mathcal{Q},\IP{,})$ be the holomorphic orthogonal bundle associated to the solution as in Lemma \ref{lem:HSQ}. Consider the identification
\begin{equation*}
\mathcal{Q} \cong TX \otimes \mathbb{C} \oplus \End V_0 \oplus \End V_1
\end{equation*}
induced by \eqref{eq:Bismutiso}. Then, the (possibly indefinite) Hermitian metric
\begin{equation}\label{eq:Gexplicit}
\mathbf{G}  =\left(\begin{array}{ccc}
    g & 0 & 0 \\
    0 & \alpha \tr_{V_0} & 0\\
    0 & 0 &  -\alpha \tr_{V_1}
    \end{array}\right)
\end{equation}
is Hermitian-Einstein with respect to $g$, that is,
\begin{equation}\label{eq:HE}
F_\mathbf{G} \wedge \omega^{n-1} = 0,
\end{equation}
where $\omega$ is the Hermitian form of $g$ and $F_\mathbf{G}$ denotes the curvature of the Chern connection of $\mathbf{G}$ on $\mathcal{Q}$.
\end{proposition}

\begin{remark}\label{rem:DelaOssa}
Proposition \ref{prop:HS-HE} shall be compared with the physical result by De la Ossa, Larfors, and Svanes in \cite[Corollary 1]{OssaLarforsSvanes}, who observed that the Hull-Strominger system is equivalent to \eqref{eq:HE} to all orders in perturbation theory.
\end{remark}

\subsection{The K\"ahler property versus slope stability}\label{sec:slope}

In this section we investigate the Mumford-Takemoto slope stability of the holomorphic orthogonal bundle $(\mathcal{Q},\IP{,})$ in relation to the K\"ahler property of the solution. In particular, we will recover a no-go result for the Hull-Strominger system, back to the seminal work of Candelas, Horowitz, Strominger, and Witten \cite{CHSW}.

Let $X$ be a compact complex manifold of dimension $n$. We assume that $X$ admits a balanced Hermitian metric $\omega_0$, that is, $d\omega_0^{n-1} = 0$. We denote by
$$
\mathfrak{b}_0 = [\omega_0^{n-1}] \in H^{n-1,n-1}_{BC}(X,\mathbb{R})
$$
the associated balanced class in Bott-Chern cohomology. Let $(\mathcal{Q},\IP{,})$ be a holomorphic orthogonal bundle over $X$. Recall that a positive definite Hermitian metric $\mathbf{H}$ on $\mathcal{Q}$ is said to be compatible with the orthogonal structure $\IP{,}$ if there exists a $\mathbb{C}$-antilinear orthogonal involution $\sigma \colon \mathcal{Q} \to \mathcal{Q}$, that is, $\langle \sigma, \sigma \rangle = \overline{\langle, \rangle}$ and $\sigma^2=id$, and such that
$$
\mathbf{H} = \IP{\cdot,\sigma \cdot}.
$$
Given such a metric $\mathbf{H}$, we will denote by $D^\mathbf{H}$ and $F_\mathbf{H} := F_{D^\mathbf{H}}$ its Chern connection and Chern curvature, respectively.

\begin{remark}\label{rem:Gcompatible}
Observe that the (possibly indefinite) Hermitian metric $\mathbf{G}$ in Proposition \ref{prop:HS-HE} is precisely of this form, for $\sigma(s) = - \overline{s}$. Here, the conjugation is obtained via the isomorphism $\mathcal{Q} \cong (TX \oplus \ad P_h) \otimes \mathbb{C}$ induced by \eqref{eq:Bismutiso}, where $P_h$ denotes the bundle of unitary frames of $(V_0\oplus V_1,h_0 \oplus h_1)$. In the sequel, we will reserve the notation $\mathbf{H}$ for Hermitian metrics which are positive definite.
\end{remark}

The existence of a compatible Hermitian metric $\mathbf{H}$ on $(\mathcal{Q},\IP{,})$ satisfying the Hermitian-Einstein equation
\begin{equation}\label{eq:HEH}
F_\mathbf{H} \wedge \omega_0^{n-1} = 0
\end{equation}
can be characterized in terms of a slope stability criteria as in the Donaldson-Uhlenbeck-Yau Theorem \cite{Don,UYau} and its extensions to Hermitian manifolds (see \cite{Buchdahl,LiYauHYM,lt}). To state the precise result which we will use, let us recall first some basic definitions. Given a torsion-free coherent sheaf $\mathcal{F}$ of $\mathcal{O}_X$-modules over $X$, the determinant $\det \mathcal F : = ((\Lambda^r \mathcal F)^*)^*$, where $r$ denotes the rank of $\mathcal{F}$, is a holomorphic line bundle over $X$. Given now the balanced class $\mathfrak{b}_0$, we can define the slope of $\mathcal{F}$ by
$$
\mu_{\mathfrak{b}_0}(\mathcal{F}) = \frac{c_1(\det \mathcal{F})\cdot \mathfrak{b}_0}{r} \in \mathbb{R}
$$ 
where the first Chern class $c_1(\det \mathcal{F})$ of $\det \mathcal{F}$ and $\mathfrak{b}_0$ are regarded as elements in the De Rham cohomology of $X$. We say that a subsheaf $\mathcal{F} \subset \mathcal{Q}$ is isotropic if $\IP{\mathcal{F},\mathcal{F}} = 0$ (see e.g. \cite{GP,Araujo}).

\begin{definition}\label{d:stab} 
Let $X$ be a compact complex manifold endowed with a balanced class $\mathfrak{b}_0 \in H_{BC}^{n-1,n-1}(X,\mathbb{R})$. A holomorphic orthogonal bundle $(\mathcal{Q},\IP{,})$ over $X$ is
\begin{enumerate}

\item  \emph{slope $\mathfrak{b}_0$-semistable} if for any isotropic coherent subsheaf $\mathcal{F} \subset \mathcal{Q}$ one has
$$
\mu_{\mathfrak{b}_0}(\mathcal{F}) \; \leqslant \; 0,
$$

\item  \emph{slope $\mathfrak{b}_0$-stable} if for any proper isotropic coherent subsheaf $\mathcal{F} \subset \mathcal{Q}$ one has
$$
\mu_{\mathfrak{b}_0}(\mathcal{F}) \; < \; 0,
$$

\item \emph{slope $\mathfrak{b}_0$-polystable} if it is slope $\mathfrak{b}_0$-semistable and whenever $\mathcal{F} \subset \mathcal{Q}$ is an isotropic coherent subsheaf with $\mu_{\mathfrak{b}_0}(\mathcal{F}) = 0$, there is a coisotropic coherent subsheaf $\mathcal{W} \subset \mathcal{Q}$ such that
$$
\mathcal{Q} = \mathcal{W} \oplus \mathcal{F}.
$$

\end{enumerate}

\end{definition}


The relation between slope stability and the Hermitian-Einstein equation \eqref{eq:HEH} for compatible Hermitian metrics is provided by the following version of the Donaldson-Uhlenbeck-Yau Theorem (see e.g. \cite{Araujo,lt}):

\begin{theorem}\label{t:DUY} Let $X$ be a compact complex manifold. Let $\omega_0$ be a balanced Hermitian metric on $X$ with balanced class $\mathfrak{b}_0 = [\omega^{n-1}_0] \in H_{BC}^{n-1,n-1}$. A holomorphic orthogonal bundle $(\mathcal{Q},\IP{,})$ over $X$ admits a compatible Hermitian metric $\mathbf{H}$ solving the Hermitian-Einstein equation \eqref{eq:HEH} if and only if it is slope $\mathfrak{b}_0$-polystable.
\end{theorem}

Let us turn next to the relation with the Hull-Strominger system. As in Section \ref{sec:HS-HE}, we assume that $X$ admits a holomorphic volume form $\Omega$. We fix a pair of holomorphic vector bundles $V_0$ and $V_1$ over our compact complex manifold $X$ satisfying \eqref{eq:c1c2abstract} and take $P$ to be the holomorphic principal bundle of split frames of $V_0 \oplus V_1$ equipped with the pairing \eqref{eq:pairingHS}. Let $(g,h_0,h_1)$ be a solution of the Hull-Strominger system \eqref{eq:HSabstract}. Consider the associated holomorphic orthogonal bundle $(\mathcal{Q},\IP{,})$. In our next results we investigate the relationship between the slope polystable of $(\mathcal{Q},\IP{,})$, in the sense of Definition \ref{d:stab}, and the K\"ahler property of the solution. The key to our argument is the existence of a canonical isotropic subsheaf given by the holomorphic cotangent bundle
\begin{equation}\label{eq:T*injects}
T^*_{1,0} \overset{\pi^*}{\hookrightarrow} \mathcal{Q}.
\end{equation}

\begin{lemma}\label{lem:HSnogo}
Let $X$ be a compact K\"ahler manifold endowed with a holomorphic volume form $\Omega$. Let $V_0$ and $V_1$ be holomorphic vector bundles over $X$ satisfying \eqref{eq:c1c2abstract}. Let $(g,h_0,h_1)$ be a solution of the Hull-Strominger system \eqref{eq:HSabstract} with $\alpha \in \RR$ and consider the associated holomorphic orthogonal bundle $(\mathcal{Q},\IP{,})$. Suppose that  $X$ admits a balanced class $\mathfrak{b}_0 \in H^{n-1,n-1}(X,\mathbb{R})$ such that $(\mathcal{Q},\IP{,})$ is slope $\mathfrak{b}_0$-polystable, and K\"ahler classes $[\omega_i]\in H^{1,1}(X,\mathbb{R})$ such that $V_i$ is $[\omega_i]$-polystable. Then $g$ is a K\"ahler metric and $h_0$ and $h_1$ are flat.
\end{lemma}

\begin{proof}
Assume first that $\alpha \neq 0$. We will prove that $(g,h_0,h_1)$ is also a solution of \eqref{eq:HSabstract} with $\alpha = 0$.   
Consider the canonical isotropic subsheaf \eqref{eq:T*injects}. The existence of a holomorphic volume form $\Omega$ implies that 
$$
\mu_{\mathfrak{b}_0}(T^*_{1,0}) = 0,
$$
for any given balanced class $\mathfrak{b}_0 \in H^{n-1,n-1}$. Hence, assuming that $\mathcal{Q}$ is slope $\mathfrak{b}_0$-polystable, we have that
$$
\mathcal{Q} = \mathcal{W} \oplus T^*_{1,0}
$$
for a holomorphic coisotropic subbundle $\mathcal{W} \subset \mathcal{Q}$. Note further that there are canonical isomorphisms of holomorphic vector bundles
$$
\mathcal{W} \cong \mathcal{Q}/T^*_{1,0} \cong A_P,
$$
and therefore the class of the extension 
\begin{equation}\label{eq:extension}
\begin{split}
0 \longrightarrow T^*_{1,0} \overset{\pi^*}{\longrightarrow} \mathcal{Q} \overset{\pi}{\longrightarrow} A_P \longrightarrow 0
\end{split}
\end{equation}
vanishes. Note that there is a biholomorphism $A_P \cong (T^{1,0} \oplus \End V_0 \oplus \End V_1, \overline{\partial}_0)$ where the Dolbeault operator on the right hand-side is 
$$
\dbar_0 (V + r_0 + r_1)  = \dbar V + i_V F_{h_0} + i_V F_{h_1} + \dbar^{V_0} r_0 + \dbar^{V_1} r_1
$$
and the class of \eqref{eq:extension} is represented by $\gamma \in \Omega^{0,1}(\Hom(A_P,T^*_{1,0}))$, defined by
$$
i_{W}\gamma (V + r_0 + r_1)  =  - i_{W}i_V(2i\partial \omega) - 2 \alpha \tr_{V_0} (i_{W} F_{h_0}r_0) + 2 \alpha \tr_{V_1} (i_{W} F_{h_1}r_1)
$$
for any $V + r_0 + r_1 \in T^{1,0} \oplus \End V_0 \oplus \End V_1$ and $W \in T^{0,1}$. Therefore, the condition 
$$
[\gamma] = 0 \in H^1(\Hom(A_P,T^*_{1,0}))
$$
jointly with $\alpha \neq 0$ implies, in particular, the existence of $a_j \in \Omega^{1,0}(\End V_j)$ such that
$$
\dbar^{V_0} a_0 = F_{h_0}, \qquad \dbar^{V_1} a_1 = F_{h_1}.
$$
By hypothesis, there exists K\"ahler classes $[\omega_i]\in H^{1,1}(X,\mathbb{R})$ such that $V_i$ is $[\omega_i]$-polystable. Let $\tilde{h}_j$ be a Hermitian-Einstein metric on $V_j$ with respect to $\omega_j$. Then, we can use the standard identity in K\"ahler geometry
$$
- \frac{8 \pi^2}{(n-2)!} ch_2(V_j) \cdot [\omega_j]^{n-2} = \|F_{\tilde{h}_j}\|^2_{L^2},
$$
where the $L^2$-norm of the curvature $F_{\tilde{h}_j}$ is calculated with respect to the metrics $\tilde{h}_j$ and $\omega_j$ using the volume form $\omega_j^n/n!$. Using that $\dbar^{V_j} a_j = F_{h_j}$, the left hand side of this expression vanishes by Chern-Weyl theory, and therefore $\tilde{h}_0$ and $\tilde{h}_1$ are flat. In particular,
$$
F_{\tilde h_j}\wedge \omega^{n-1}=0
$$
and hence, since $h_j$ must be related to $\tilde h_j$ by a holomorphic gauge transformation, $h_j$ are also flat. Therefore, $(g,h_0,h_1)$ solves \eqref{eq:HSabstract} with $\alpha = 0$, and the Bianchi identity reads
$$
dd^c\omega=0.
$$
By \cite{LiYau}, the conformally balanced equation is equivalent to
$$
d^c(\mathrm{log} \|\Omega\|_\omega)-d^*\omega=0
$$
and, by \cite[Proposition 4.11]{GFGM}, this implies
$$
\nabla^B(\|\Omega\|^{-1}_\omega \Omega)=0.
$$
Then, since $-i\rho_{B}(\omega)$ is the induced curvature of $\nabla^B$ on the anti-canonical bundle $K_X^{-1}$, it follows that $\rho_B(\omega)=0$. Thus, applying \cite[Theorem 4.7]{GaJoSt} it follows from the existence of a holomorphic volume form $\Omega$ that $g$ is K\"ahler. 

\end{proof}

\begin{remark}
Notice that the proof of \cite[Theorem 4.7]{GaJoSt}, which we have used to conclude that $g$ is K\"ahler, uses a slope stability argument via exact holomorphic Courant algebroids. Therefore, our proof of Lemma \ref{lem:HSnogo} reduces to Geometric Invariant Theory.
\end{remark}

Our next result provides an obstruction to the existence of non-K\"ahler solutions of the Hull-Strominger system \eqref{eq:HSabstract}. It is a direct consequence of Lemma \ref{lem:HSnogo}.


\begin{proposition}\label{prop:HSnogo}
Let $X$ be a compact K\"ahler manifold endowed with a holomorphic volume form $\Omega$. Let $V_0$ and $V_1$ be holomorphic vector bundles over $X$ satisfying \eqref{eq:c1c2abstract}. Let $(g,h_0,h_1)$ be a solution of the Hull-Strominger system \eqref{eq:HSabstract} with $\alpha \in \RR$ and consider the associated holomorphic orthogonal bundle $(\mathcal{Q},\IP{,})$. Suppose that $(\mathcal{Q},\IP{,})$ is slope polystable with respect to $\mathfrak{b}:= [\|\Omega\|_\omega \omega^{n-1}]$ and furthermore that $\mathfrak{b}$ is the $(n-1)^{\mathrm{th}}$-power of a K\"ahler class, then $g$ is K\"ahler and $h_0$ and $h_1$ are flat.
\end{proposition}

\begin{remark}
Our previous result applies, in particular, to the solutions of the Hull-Strominger system found recently by Collins, Picard, and Yau in \cite[Section 3.2]{CoPiYau2}. These solutions are on a K\"ahler Calabi-Yau threefold, have $V_0$ isomorphic to $T^{1,0}$, and are constructed such that $\mathfrak{b}$ can be prescribed to be the square of any given K\"ahler class. Given that $T^{1,0}$ has non-trivial Chern classes (e.g., when $X$ is simply connected), Proposition \ref{prop:HSnogo} proves that the associated $(\mathcal{Q},\IP{,})$ is not slope $\mathfrak{b}$-polystable in this case.
\end{remark}

\begin{remark}
Observe that the proof of Proposition \ref{prop:HSnogo} via Lemma \ref{lem:HSnogo} uses crucially the K\"ahler hypothesis of the manifold. This poses the question of whether on a \textit{a priori} general compact complex manifold, a solution to the Hull-Strominger system with $\mathfrak{b}$-polystable orthogonal bundle $(\mathcal{Q},\IP{,})$ must be K\"ahler. Note that $\mathfrak{b}$-polystability implies that $\mathcal{Q}$ holomorphically splits as $T_{1,0}^* \oplus A_P$. This condition can be thought of as an analogue for metrics on Bott-Chern algebroids (see \cite{GFGM}) of the Hermitian symplectic condition for pluriclosed metrics in the Streets-Tian's Conjecture.
\end{remark}

We consider next the special case that the Hermitian-Einstein metric $\mathbf{G}$ on $\mathcal{Q}$ associated to our solution is positive definite (see Proposition \ref{prop:HS-HE}). Without loss of generality, we can assume that $\alpha > 0$. Thus, by construction, the metric $\mathbf{G}$ is positive definite precisely when $\rk V_0 = 0$.
Specifying to the case of complex dimension three, we recover a no-go result for the original Hull-Strominger system  back to the seminal work of Candelas, Horowitz, Strominger, and Witten \cite{CHSW}. This shows the necessity of introducing the connection $\nabla$ for the existence of non-K\"ahler solutions (see Remark \ref{rem:HSabstract}).

\begin{proposition}\label{cor:HSnogo}
Let $X$ be compact complex manifold endowed with a holomorphic volume form $\Omega$. Let $V$ be holomorphic vector bundle over $X$ satisfying $ch_2(V) = 0 \in H^{2,2}_{BC}(X,\RR)$. Let $(g,h)$ be a solution of the system
\begin{equation}\label{eq:HSabstractbis}
\begin{split}
F_{h}\wedge \omega^{n-1} & =0,\\
d(\|\Omega\|_\omega \omega^{n-1}) & = 0,\\
dd^c \omega + \alpha \tr F_{h} \wedge F_{h} & = 0,
\end{split}
\end{equation}
with $\alpha > 0$. Then, $g$ is K\"ahler and $h$ is flat.
\end{proposition}

\begin{proof}
Consider the holomorphic orthogonal bundle $(\mathcal{Q},\IP{,})$ with the Hermitian-Einstein metric $\mathbf{G}$ associated to the solution $(g,h)$ as in Proposition \ref{prop:HS-HE}. Via the identification
\begin{equation*}
\mathcal{Q} \cong TX \otimes \mathbb{C} \oplus \End V   
\end{equation*}
we have the explicit formula (cf. \eqref{eq:Gexplicit})
$$
\mathbf{G}  =\left(\begin{array}{ccc}
    g & 0 \\
    0 & - \alpha \tr_{V}
    \end{array}\right)
$$
and therefore $\mathbf{G}$ is defines a compatible, positive definite, Hermitian-Einstein metric on $(\mathcal{Q},\IP{,})$ with respect to the balanced metric 
$$
\omega' = \|\Omega\|_\omega^{\frac{1}{n-1}}\omega. 
$$
From Theorem \ref{t:DUY}, $(\mathcal{Q},\langle\cdot,\cdot\rangle)$ is $\mathfrak{b}$-polystable for $\mathfrak{b}=[||\Omega||_{\omega}\omega^{n-1}]$. Hence, $\mu_{\mathfrak{b}}(T^*_{1,0})=\mu_{\mathfrak{b}}(\mathcal{Q})=0$ implies that $\mathcal{Q}\cong T^*_{1,0}\oplus A_P$ holomorphically and metrically with respect to the metric $\mathbf{G}$. This means that the second fundamental form of the extension
$$
0\rightarrow T^*_{1,0}\rightarrow \mathcal{Q}\rightarrow A_P \rightarrow 0
$$
given by $\gamma\in \Omega^{0,1}(\mathrm{Hom}(A_P,T^*_{1,0}))$ as
$$
i_W\gamma(V+r)=-i_Wi_V(2i\partial\omega)+2\alpha \mathrm{tr}_V(F_h r)
$$
must vanish identically. Therefore, the result follows.
\end{proof}

\section{Harmonic metrics, Higgs fields, and stability}\label{sec:Higgs}

\subsection{HyperK\"ahler moment maps}\label{sec:KKmmap}

Let $(X,g)$ be a compact complex manifold of dimension $n$ endowed with a hermitian metric $g$. We assume that $g$ is balanced, that is, $\omega = g(J,)$ satisfies $d \omega^{n-1} = 0$. We fix a smooth complex vector bundle $Q$ over $X$ of degree zero 
$$
c_1(Q) \cdot [\omega^{n-1}] = 0.
$$
We are interested in the geometry of the space of complex connections on $Q$, which we denote by $\cA_Q$. In the application to Section \ref{sec:HarmonicMet}, $Q$ is a complex orthogonal bundle and $\cA_Q$ is replaced by the space of orthogonal complex connections. Nonetheless, the setup discussed here applies with minor modifications and hence we stick to the simpler situation stated above.

The infinite-dimensional space of complex connections $\cA_Q$ is affine, modelled on the complex vector space
$$
\Omega^1(\End \; Q).
$$
It is endowed with a natural complex symplectic structure, defined by
$$
\Omega_\CC(a_1^c,a_2^c) = - \int_X \tr a_1^c \wedge a_2^c \wedge \frac{\omega^{n-1}}{(n-1)!}.
$$
The group of complex gauge transformations $\cG_Q$ of $Q$ acts on $\cA_Q$ by symplectomorphisms and, similarly as in the Atiyah-Bott-Donaldson picture, there is a complex moment map. The proof is an exercise using the balanced condition of the metrics, and is left to the reader.

\begin{lemma}\label{lemApp:Cmmap}
Assume that the Hermitian form $\omega$ is balanced. Then, the $\cG_Q$-action on $(\cA_Q,\Omega_\CC)$ is Hamiltonian with moment map
$$
\IP{\mu_\CC(D),s^c} = - \int_X \tr s^c F_D \wedge \frac{\omega^{n-1}}{(n-1)!},
$$
where $s^c \in \Omega^0(\End Q) \cong \Lie \cG_Q$ and $F_D$ denotes the curvature of $D$.
\end{lemma}

Observe that the zeros of the complex moment map are given by connections $D \in \cA_Q$ satisfying
$$
F_D \wedge \omega^{n-1} = 0.
$$
One can restrict the Hamiltonian $\cG_Q$-action to the complex subspace $\cA_Q^{1,1} \subset \cA_Q$ given by connections with $F_D^{0,2} = F_D^{2,0} = 0$, obtaining a complex analogue of the Hermitian-Yang-Mills equations, similarly as in Proposition \ref{prop:HS-HE} (cf. \cite{KaledinVerbitsky}).

To introduce the hyperK\"ahler structure on $\cA_Q$ of our interest, following \cite{Hitchin1987} we fix a positive definite hermitian metric $\mathbf{H}$ on $Q$. Then, given $D \in \cA_Q$ there is a unique decomposition
\begin{equation}\label{eq:H-decompositionapp}
D = \nabla^\mathbf{H} + \Psi,
\end{equation}
where $\nabla^\mathbf{H}$ is an $\mathbf{H}$-unitary connection and $\Psi \in i\Omega^{1}(\End_{\mathbf{H}} Q)$, where 
$$
\Omega^{1}(\End_{\mathbf{H}} Q) := \{ a \in \Omega^{1}(\End Q)\; | \; 
a^{*_\mathbf{H}} = - a\}.
$$
This induces an identification
$$
\cA_Q = \cA_{\mathbf{H}} \times i\Omega^{1}(\End_{\mathbf{H}} Q)
$$
and a decomposition
$$
\Omega_\CC = \Omega_\mathbf{I} + i \Omega_\mathbf{J}
$$
where
\begin{align*}
\Omega_\mathbf{I}(a_1^c,a_2^c) & = - \int_X \tr a_1 \wedge a_2 \wedge \frac{\omega^{n-1}}{(n-1)!} -  \int_X \tr \psi_1 \wedge \psi_2 \wedge \frac{\omega^{n-1}}{(n-1)!} \\
\Omega_\mathbf{J}(a_1^c,a_2^c) & =  i \int_X \tr \psi_1 \wedge a_2 \wedge \frac{\omega^{n-1}}{(n-1)!} + i \int_X \tr a_1 \wedge \psi_2 \wedge \frac{\omega^{n-1}}{(n-1)!}
\end{align*}
for $a_j^c = a_j + \psi_j$. From this, using the fact that the base manifold has a complex structure $J$, one can infer a hyperK\"ahler structure with metric
\begin{align*}
g(a^c,a^c) & = - \int_X \tr a \wedge *_\omega a +  \int_X \tr \psi \wedge *_\omega \psi
\end{align*}
and complex structures $\mathbf{I}, \mathbf{J}, \mathbf{K}$, satisfying $\mathbf{I}\mathbf{J}\mathbf{K} = \mathbf{I}^2 = \mathbf{J}^2 = \mathbf{K}^2 = -\Id$, defined by
\begin{align*}
\mathbf{I} a^c = J a - J\psi, \qquad \mathbf{J} a^c = -iJ \psi + i Ja, \qquad \mathbf{K} a^c = i\psi + i a.
\end{align*}
We are interested in the Hamiltonian action of the unitary gauge group $\cG_\mathbf{H} \subset \cG_Q$ for the triple of symplectic structures $\Omega_\mathbf{I}, \Omega_\mathbf{J}, \Omega_\mathbf{K}$, where $\Omega_\mathbf{K} := g(\mathbf{K}\cdot,\cdot)$ is given by
$$
\Omega_\mathbf{K}(a^c_1,a^c_2) = -i\int_X  \tr \psi_1 \wedge J a_2 \wedge \frac{\omega^{n-1}}{(n-1)!} + i  \int_X \tr a_1 \wedge J\psi_2 \wedge \frac{\omega^{n-1}}{(n-1)!}.
$$
\begin{proposition}\label{propApp:HKmmap}
Assume that $\omega$ is balanced. Then, the $\cG_\mathbf{H}$-action on $\cA_Q$ is Hamiltonian for the three symplectic structures $\Omega_\mathbf{I}, \Omega_\mathbf{J}, \Omega_\mathbf{K}$, and there is a hyperK\"ahler moment map
$$
\mu = (\mu_\mathbf{I}, \mu_\mathbf{J}, \mu_\mathbf{K})
$$
where
\begin{align*}
\IP{\mu_\mathbf{I}(D),s} & =  - \int_X \tr s  (F_{\nabla^{\mathbf{H}}} + \tfrac{1}{2}[\Psi \wedge \Psi]) \wedge \frac{\omega^{n-1}}{(n-1)!},\\
\IP{\mu_\mathbf{J}(D),s} & =  i \int_X \tr s  \nabla^{\mathbf{H}} \Psi  \wedge \frac{\omega^{n-1}}{(n-1)!},\\
\IP{\mu_\mathbf{K}(D),s} & =  i \int_X \tr s \nabla^{\mathbf{H}} (J\Psi) \wedge \frac{\omega^{n-1}}{(n-1)!},
\end{align*}
and $s \in \Omega^{0}(\End_{\mathbf{H}} Q) \cong \Lie \cG_\mathbf{H}$.
\end{proposition}

\begin{proof}
Equation \eqref{eq:H-decompositionapp} implies that
\begin{equation}\label{eq:FDdecomp}
F_{\mathbf{G}} = F_{\nabla^{\mathbf{H}}} + \nabla^{\mathbf{H}} \Psi + \tfrac{1}{2}[\Psi \wedge \Psi].
\end{equation}
Then, the formula for $\mu_\mathbf{I}$ and $\mu_\mathbf{J}$ follow from Lemma \ref{lemApp:Cmmap} by taking real and imaginary parts in the formula for $\mu_\CC$. The fact that $\mu_\mathbf{K}$ is a moment map follows easily from the explicit expression for $\Omega_\mathbf{K}$ above. 
\end{proof}

To finish this section, we give a characterization of the hyperK\"ahler moment map equations $\mu(D) = 0$ for a complex connection $D = \nabla^{\mathbf{H}} + \Psi$, given by
\begin{equation}\label{eq:IJKmomentmapApp}
\begin{split}
(F_{\nabla^{\mathbf{H}}} + \tfrac{1}{2}[\Psi \wedge \Psi]) \wedge \omega^{n-1}  & = 0,\\
(\nabla^{\mathbf{H}} \Psi) \wedge \omega^{n-1} & = 0,\\
(\nabla^{\mathbf{H}} J\Psi) \wedge \omega^{n-1} & = 0.\\
\end{split}
\end{equation}
To link with the definition of a harmonic metric in Section \ref{sec:HarmonicMet}, it is convenient to remove our assumption that the Hermitian metric $\omega$ is balanced (in our applications, the Hermitian metric is conformally balanced). 

\begin{lemma}\label{lemma:IJmmapApp}
Let $\omega$ be an arbitrary Hermitian form on $X$. Then, a complex connection $D = \nabla^{\mathbf{H}} + \Psi$ satisfies \eqref{eq:IJKmomentmapApp} if and only if
\begin{equation}\label{eq:IJKharmonicApp}
\begin{split}
(F_{\nabla^{\mathbf{H}}} + \tfrac{1}{2}[\Psi \wedge \Psi]) \wedge \omega^{n-1}  & = 0,\\
(\nabla^{\mathbf{H}})^{*} (J\Psi) - i_{J \theta_\omega^\sharp}\Psi & = 0,\\
(\nabla^{\mathbf{H}})^* \Psi + i_{ \theta_\omega^\sharp}\Psi & = 0.
\end{split}
\end{equation}
where $\theta_\omega = Jd^*\omega$ is the Lee form of $\omega$ and $(\nabla^{\mathbf{H}})^{*}$ denotes the adjoint operator with respect to the metric $g = \omega(\cdot,J\cdot)$.
\end{lemma}

\begin{proof}
The statement follows easily from the formula
\begin{equation}\label{eq:harmonicvsmmap}
(\nabla^{\mathbf{H}})^* (J\Psi) = \tfrac{1}{(n-1)!}* (\nabla^{\mathbf{H}} \Psi)  \wedge \omega^{n-1} + i_{J \theta_\omega^\sharp}\Psi.
\end{equation}
\end{proof}

The third equation in \eqref{eq:IJKharmonicApp}, corresponding to the condition $\mu_\mathbf{K}(D) = 0$ when $\omega$ is conformally balanced, will be taken in Section \ref{sec:HarmonicMet} as the defining equation for our notion of harmonic metric for the Hull-Strominger system.

\subsection{Harmonic metrics}\label{sec:HarmonicMet}

We introduce next our notion of harmonic metric for the Hull-Strominger system, motivated by the hyperK\"ahler moment map construction in the previous section. We fix $(X,\Omega)$ and $V_0, V_1$ as in Section \ref{sec:HS-HE}. Let $(g,h_0,h_1)$ be a solution of the Hull-Strominger system \eqref{eq:HSabstract} and consider the associated holomorphic orthogonal bundle $(\mathcal{Q},\IP{,})$. We are mainly interested in non-K\"ahler solutions, and therefore we will assume that $\alpha > 0$ and $\rk V_0 > 0$ (see Corollary \ref{cor:HSnogo}). Consequently, the generalized Hermitian metric $\mathbf{G}$ associated to our solution will be indefinite (see \eqref{eq:Gexplicit}). 

The fundamental object in our development is the orthogonal connection $D^\mathbf{G}$ on  $(\mathcal{Q},\IP{,})$ in Proposition \ref{prop:HS-HE}. Explicitly, in matrix notation in terms of the identification 
\begin{equation}\label{eq:Qascomplexified}
\mathcal{Q} \cong TX \otimes \mathbb{C} \oplus \End V_0 \oplus \End V_1,
\end{equation}
for any vector field $V$ the operator $D^{\mathbf{G}}_V$ is given by (see \cite[Proposition 3.4]{GFGM})
\begin{equation}\label{eq:DGexp}
D^{\mathbf{G}}_V = \left(\begin{array}{ccc}
    \nabla^{-}_V & g^{-1}\alpha\tr (i_VF_{h_0}\cdot) &  -g^{-1}\alpha\tr (i_VF_{h_1} \cdot) \\
     - F_{h_0}(V,\cdot)  & d^{h_0}_V & 0\\
    - F_{h_1}(V,\cdot) & 0 & d^{h_1}_V
    \end{array}\right),
\end{equation}
where $\nabla^-$ denotes the $\CC$-linear extension of the $g$-compatible connection with totally skew-symmetric torsion $d^c \omega$, that is,
$$
\nabla^- = \nabla + \tfrac{1}{2}g^{-1}d^c\omega
$$
for $\nabla$ the Levi-Civita connection of $g$. Given a compatible (positive definite) Hermitian metric $\mathbf{H}$ on the holomorphic orthogonal bundle $(\mathcal{Q},\IP{,})$ there exists a unique decomposition
\begin{equation}\label{eq:H-decomposition}
D^\mathbf{G} = \nabla^\mathbf{H} + \Psi,
\end{equation}
where $\nabla^\mathbf{H}$ is an $\mathbf{H}$-unitary connection and $\Psi \in \Omega^{1}(\End \mathcal{Q})$ satisfies 
$$
\Psi^{*_\mathbf{H}} = \Psi.
$$

\begin{lemma}\label{lemma:IJmmap}
The pair $(\nabla^{\mathbf{H}},\Psi)$ in \eqref{eq:H-decomposition} satisfies the equations
\begin{equation}\label{eq:IJmomentmap}
\begin{split}
(F_{\nabla^{\mathbf{H}}} + \tfrac{1}{2}[\Psi \wedge \Psi]) \wedge \omega^{n-1}  & = 0,\\
(\nabla^{\mathbf{H}})^{*} (J\Psi) - i_{J \theta_\omega^\sharp}\Psi & = 0,\\
F_{\nabla^{\mathbf{H}}}^{0,2} + (\nabla^{\mathbf{H}})^{0,1} \Psi^{0,1} + \tfrac{1}{2}[\Psi^{0,1} \wedge \Psi^{0,1}]  & = 0.
\end{split}
\end{equation}
where $(\nabla^{\mathbf{H}})^{*}$ denotes the adjoint operator with respect to the metric $g$.
\end{lemma}

\begin{proof}
By Proposition \ref{prop:HS-HE}, $\mathbf{G}$ is Hermitian-Einstein. Decomposing $F_{\mathbf{G}}$ into its Hermitian and skew-Hermitian components with respect to $\mathbf{H}$ as in the proof of Proposition \ref{propApp:HKmmap}, the proof follows easily from Equation \eqref{eq:FDdecomp} and the proof of Lemma \ref{lemma:IJmmapApp}.
\end{proof}

Given that the Hermitian form $\omega$ is conformally balanced, the first and second equations in \eqref{eq:IJmomentmap} correspond to the zeros of an infinite-dimensional complex moment map in the space of complex orthogonal connections on $(\mathcal{Q},\IP{,})$ (see Section \ref{sec:KKmmap}). Similarly as in the theory of \emph{Higgs bundles} \cite{Hitchin1987}, it is therefore very natural to supplement these conditions with an additional equation arising from a hyperK\"ahler moment map (see Proposition \ref{propApp:HKmmap} and Lemma \ref{lemma:IJmmapApp}).

\begin{definition}\label{def:harmonic}
Let $(\mathcal{Q},\IP{,},D)$ be a holomorphic orthogonal bundle over a Hermitian manifold $(X,g)$ endowed with an orthogonal connection $D$ such that $D^{0,1} = \dbar_{\mathcal{Q}}$ and $F_D \wedge \omega^{n-1} = 0$. We say that compatible Hermitian metric $\mathbf{H}$ on $(\mathcal{Q},\IP{,})$ is \emph{harmonic} if
\begin{equation}\label{eq:Kmomentmap}
(\nabla^{\mathbf{H}})^* \Psi + i_{ \theta_\omega^\sharp}\Psi = 0,
\end{equation}
where we use the decomposition \eqref{eq:H-decomposition}.
\end{definition}


\begin{remark}\label{rem:Kaledin}
Notice that the connection $D$ in the previous definition is, in particular, a \emph{non-Hermitian-Yang-Mills connection} in the sense of Kaledin and Verbitsky \cite{KaledinVerbitsky}. It would be interesting to find further relations between our picture and the theory proposed in this reference.
\end{remark}

Our stability condition for the Hull-Strominger system is related to the existence of a harmonic metric on $(\mathcal{Q},\IP{,},D^{\mathbf{G}})$. We postpone its study to Section \ref{sec:higgsstb}. Here, we propose to address the following problem:

\begin{question}\label{questionharmonic}
Let $(g,h_0,h_1)$ be a solution of the Hull-Strominger system \eqref{eq:HSabstract} and consider the associated holomorphic orthogonal bundle $(\mathcal{Q},\IP{,})$ and orthogonal connection $D^\mathbf{G}$ as in Proposition \ref{prop:HS-HE}. Does $(\mathcal{Q},\IP{,},D^\mathbf{G})$ admit a harmonic metric $\mathbf{H}$?
\end{question}

In order to provide a non-trivial example of harmonic metric for $(\mathcal{Q},\IP{,},D^\mathbf{G})$ in Section \ref{subsec:exhar}, we calculate next equation \eqref{eq:Kmomentmap} for a particular choice of Hermitian metric. Via the identification \eqref{eq:Qascomplexified}, we define (cf. \eqref{eq:Gexplicit})
\begin{equation}\label{eq:H}
\mathbf{H}  =\left(\begin{array}{ccc}
    g & 0 & 0 \\
    0 & - \alpha \tr_{V_0} & 0\\
    0 & 0 &  -\alpha \tr_{V_1}
    \end{array}\right).
\end{equation}
It is not difficult to see that \eqref{eq:H} defines a compatible Hermitian metric on $(\mathcal{Q},\IP{,})$. In our next result we calculate the decomposition \eqref{eq:H-decomposition} for this particular choice of Hermitian metric.

\begin{lemma}\label{lemma:H-decompositionexp}
Let $(\nabla^\mathbf{H},\Psi)$ be the pair in \eqref{eq:H-decomposition} associated to the compatible Hermitian metric \eqref{eq:H}. Then, in matrix notation in terms of the identification \eqref{eq:Qascomplexified}, one has
\begin{align*}
\nabla^{\mathbf{H}} =\left(\begin{array}{ccc}
    \nabla^{-} & 0 &  \mathbb{F}_{h_1}^\dagger \\
      0 & d^{h_0} & 0\\
    - \mathbb{F}_{h_1} & 0 & d^{h_1}
    \end{array}\right), \qquad
\Psi = \left(\begin{array}{ccc}
    0 & \mathbb{F}_{h_0}^\dagger & 0 \\
     - \mathbb{F}_{h_0} & 0 & 0\\
    0 & 0 & 0
    \end{array}\right).
\end{align*}
where $\mathbb{F}_{h_j} \in \Omega^1(\Hom(TX \otimes \mathbb{C},\End V_j))$ are the $\Hom(TX \otimes \CC,\End V_j)$-valued $1$-forms defined by
\begin{equation}\label{eq:Foperator}
(i_V\mathbb{F}_{h_j})(W) := F_{h_j}(V,W)
\end{equation}
and $\mathbb{F}_{h_j}^\dagger$ denote the corresponding adjoints with respect to $\mathbf{G}$, that is,
$$
i_V\mathbb{F}_{h_0}^\dagger (r_0) =  g^{-1}\alpha\tr (i_VF_{h_0}r_0), \qquad i_V\mathbb{F}_{h_1}^\dagger (r_1) =  - g^{-1}\alpha\tr (i_VF_{h_1}r_1)
$$
\end{lemma}

\begin{proof}
It is not difficult to see that $\nabla^{\mathbf{H}}$, as defined above, is $\mathbf{H}$-unitary and furthermore that $\Psi^{*_\mathbf{H}} = \Psi$. The statement follows from formula \eqref{eq:DGexp}.
\end{proof}

The desired characterization of the harmonicity of \eqref{eq:H} is as follows:

\begin{lemma}\label{lemma:Kmmap}
Let $(\nabla^\mathbf{H},\Psi)$ be the pair in \eqref{eq:H-decomposition} associated to the compatible Hermitian metric $\mathbf{H}$ defined by \eqref{eq:H}. Then,
\begin{equation}\label{eq:Kmomentmapexp}
(\nabla^{\mathbf{H}})^* \Psi + i_{ \theta_\omega^\sharp}\Psi = \left(\begin{array}{ccc}
   0 & - \mathbb{U}^\dagger &  0 \\
\mathbb{U}  & 0 & - \mathbb{V}^\dagger\\
    0 & \mathbb{V} &  0
    \end{array}\right).
\end{equation}
where
\begin{align*}
\mathbb{U}(V) & =-i_{V}\Bigg{(}d^{h_0 *} F_{h_0} + i_{\theta_\omega^\sharp}F_{h_0} + *(F_{h_0}\wedge * d^c\omega) \Bigg{)},\\
\mathbb{V}(r_0) & = \alpha F_{h_1}(e_i,e_j)\tr(F_{h_0}(e_i,e_j)r_0) 
\end{align*}
for any choice of $g$-orthonormal frame  $e_1, \ldots, e_{2n}$ of $T$, and $\mathbb{U}^\dagger$ and $\mathbb{V}^\dagger$ denote the corresponding adjoints with respect to $\mathbf{G}$. Consequently, $\mathbf{H}$ is harmonic if and only if the following conditions are satisfied
\begin{equation}\label{eq:harmonicHexp}
F_{h_0}\wedge * d^c\omega = 0, \qquad \alpha F_{h_1}(e_i,e_j)\tr(F_{h_0}(e_i,e_j) \cdot) = 0.
\end{equation}
\end{lemma}

\begin{proof}
By the identity $\Psi^{*_\mathbf{H}} = \Psi$, it suffices to calculate $\mathbb{U}$ and $\mathbb{V}^\dagger$. For this, we compute
\begin{align*}
(\nabla^{\mathbf{H}})^* \Psi(V) & =  - i_{e_i} (\nabla^{\mathbf{H,g}}_{e_i} \Psi)(V)\\
& = - (\nabla^{\mathbf{H}}_{e_i} (i_{e_i}\Psi(V))  - i_{e_i}\Psi(\nabla^{\mathbf{H}}_{e_i} V) - i_{\nabla_{e_i}e_i} \Psi(V))\\
& = d^{h_0}_{e_i}(F_{h_0}(e_i,V)) - F_{h_0}(e_i,\nabla^-_{e_i}V) - F_{h_0}(\nabla_{e_i}e_i,V) \\
& = - i_V d^{h_0 *} F_{h_0} - \tfrac{1}{2} F_{h_0}(e_i,g^{-1}i_V i_{e_i} d^c\omega)\\
& = - i_V (d^{h_0 *} F_{h_0} + * (F_{h_0} \wedge * d^c \omega)),\\
(\nabla^{\mathbf{H}})^* \Psi(r_1) & = i_{e_i}\Psi(\nabla^{\mathbf{H}}_{e_i} r_1) = \alpha F_{h_0}(e_i,e_j)\tr(F_{h_1}(e_i,e_j) r_1).
\end{align*}
Formula \eqref{eq:Kmomentmapexp} follows now from the explicit formula for $\Psi$ in Lemma \ref{lemma:H-decompositionexp}. The last part of the statement follows from \cite[Lemma 4.3]{GFGM}, which proves that the Hermitian-Einstein equation for $h_0$
$$
F_{h_0}\wedge \omega^{n-1}=0
$$
implies, in particular,
$$
d^{h_0 *} F_{h_0} + i_{\theta_\omega^\sharp}F_{h_0} - *(F_{h_0}\wedge * d^c\omega)=0.
$$
\end{proof}

\begin{remark}\label{rem:harmonicHexp}
Geometrically, the condition \eqref{eq:harmonicHexp} means that the two-form components of $F_{h_0}$ are orthogonal to the two-form components of the torsion $g^{-1}d^c\omega$ and also to the two-form components of the curvature $F_{h_1}$.
\end{remark}

\subsection{Stability and Higgs fields}\label{sec:higgsstb}

The moment map constructions in \cite{Waldram,GaRuTi3} suggest that the Hull-Strominger system is related to a stability condition in the sense of Geometric Invariant Theory. As we have seen in Proposition \ref{prop:HSnogo}, the naive guess of considering slope polystability of the orthogonal bundle $(\mathcal{Q},\IP{,})$ with respect to the balanced class of the solution does not work. We propose next a refined stability condition based on the existence of harmonic metrics for the Hull-Strominger system. Even though our picture is mostly conjectural, we expect that this stability condition will lead us to new obstructions to the existence of solutions in future studies.

In order to relate the existence of a harmonic metric in the sense of Definition \ref{def:harmonic} with a numerical stability condition, we introduce the following technical definition.

\begin{definition}
Let $X$ be a complex manifold, $\mathcal{Q}$ a holomorphic vector bundle over $X$, and $\mathcal{F}\subset \mathcal{Q}$ a coherent subsheaf  of $\mathcal{O}_X$-modules with singularity set $S\subset X$, \textit{i.e.} $S$ is minimal such that $\mathcal{F}|_{X\backslash S}$ is locally free. Then, given a (smooth complex) connection $D$ on $\mathcal{Q}$, we say that $\mathcal{F}$ is preserved by $D$ if
$$
D(\mathcal{F}|_{X\backslash S}) \subset \Omega^1(X\backslash S, \mathcal{F}|_{X\backslash S}). 
$$
\end{definition}

\begin{remark}
Observe that any coherent subsheaf of a vector bundle is in particular torsion-free, so the notions of degree and slope introduced in Section \ref{sec:Higgs} apply.
\end{remark}

The stability condition of our interest, is for tuples $(\mathcal{Q},\IP{,},D)$, where $(\mathcal{Q},\IP{,})$ is a holomorphic orthogonal bundle and $D$ is a (smooth) orthogonal connection such that $D^{0,1} = \dbar_{\mathcal{Q}}$, as follows (cf. \cite[Definition 8.3]{KaledinVerbitsky}).

\begin{definition}\label{d:HKstab}
Let $(X,g)$ be a compact complex manifold $X$ endowed with a balanced Hermitian metric $g$ with balanced class $\mathfrak{b} \in H^{n-1,n-1}_{BC}(X,\mathbb{R})$. Let $(\mathcal{Q},\IP{,},D)$ be a holomorphic orthogonal bundle over $X$ endowed with an orthogonal connection $D$ such that $D^{0,1} = \dbar_{\mathcal{Q}}$. We say that $(\mathcal{Q},\IP{,},D)$ is
\begin{enumerate}

\item  \emph{slope $\mathfrak{b}$-semistable} if for any isotropic coherent subsheaf $\mathcal{F} \subset \mathcal{Q}$ that is preserved by $D$ one has
$$
\mu_{\mathfrak{b}}(\mathcal{F}) \; \leqslant \; 0,
$$

\item  \emph{slope $\mathfrak{b}$-stable} if for any proper isotropic coherent subsheaf $\mathcal{F} \subset \mathcal{Q}$ that is preserved by $D$ one has
$$
\mu_{\mathfrak{b}}(\mathcal{F}) \; < \; 0,
$$

\item \emph{slope $\mathfrak{b}$-polystable} if it is slope $\mathfrak{b}$-semistable and whenever $\mathcal{F} \subset \mathcal{Q}$ is a isotropic coherent subsheaf that is preserved by $D$
with $\mu_{\mathfrak{b}}(\mathcal{F}) = 0$, there is a coisotropic subsheaf $\mathcal{W} \subset \mathcal{Q}$ that is $D$-preserved and
$$
\mathcal{Q} = \mathcal{W} \oplus \mathcal{F}.
$$

\end{enumerate}
\end{definition}

The relation between the existence of harmonic metrics, in the sense of Definition \ref{def:harmonic}, and slope stability is given in our next result.

\begin{proposition}\label{prop:harmHKstab}
Let $(X,g)$ be a compact complex manifold $X$ endowed with a balanced Hermitian metric $g$ with balanced class $\mathfrak{b} \in H^{n-1,n-1}_{BC}(X,\mathbb{R})$. Let $(\mathcal{Q},\IP{,},D)$ be a holomorphic orthogonal bundle over $X$ endowed with an orthogonal connection $D$ such that $D^{0,1} = \dbar_{\mathcal{Q}}$ and satisfying
$$
F_D\wedge \omega^{n-1}=0.
$$
Assume that $(\mathcal{Q},\IP{,},D)$ admits a harmonic metric $\mathbf{H}$. Then, $(\mathcal{Q},\IP{,},D)$ is slope $\mathfrak{b}$-polystable.
\end{proposition}

In order to prove Proposition \ref{prop:harmHKstab}, we consider the following decomposition of our connection $D$. Given a compatible Hermitian metric $\mathbf{H}$ on $(\mathcal{Q},\IP{,})$ we can uniquely write 
\begin{equation}\label{eq:C-decomposition}
D = D^\mathbf{H} + \phi,
\end{equation}
where $D^{\mathbf{H}}$ denotes the Chern connection of $\mathbf{H}$ and $\phi$ is a \emph{Higgs field}
$$
\phi \in \Omega^{1,0}(\End \mathcal{Q}).
$$

\begin{lemma}\label{lemma:IJKmmapHiggs}
Let $(\mathcal{Q},\IP{,},D)$ be a holomorphic orthogonal bundle over $X$ endowed with an orthogonal connection $D$ such that $D^{0,1} = \dbar_{\mathcal{Q}}$ and satisfying
$$
F_D\wedge \omega^{n-1}=0.
$$
Then, the pair $(D^\mathbf{H},\phi)$ in \eqref{eq:C-decomposition} satisfies the equations
\begin{equation}\label{eq:IJmomentmapHiggs}
\begin{split}
(F_{\mathbf{H}} + \tfrac{1}{2}\dbar_{\mathcal{Q}} \phi - \tfrac{1}{2}\partial^{\mathbf{H}} \phi^{*_\mathbf{H}}) \wedge \omega^{n-1}  & = 0,\\
(\dbar_{\mathcal{Q}} \phi + \partial^{\mathbf{H}} \phi^{*_\mathbf{H}}) \wedge \omega^{n-1} & = 0,\\
\partial^{\mathbf{H}} \phi + \tfrac{1}{2}[\phi \wedge \phi]  & = 0,
\end{split}
\end{equation}
where $F_{\mathbf{H}}$ denotes the Chern curvature of $\mathbf{H}$. Futhermore, $\mathbf{H}$ is harmonic if and only if
\begin{equation}\label{eq:KmomentmapHiggs}
(F_{\mathbf{H}} + \tfrac{1}{2}[\phi \wedge \phi^{*_\mathbf{H}}]) \wedge \omega^{n-1} = 0.
\end{equation}
\end{lemma}

\begin{proof}
Taking the $\mathbf{H}$-unitary part in the expression \eqref{eq:C-decomposition}, one can easily see that
$$
\nabla^{\mathbf{H}} = D^{\mathbf{H}} + \tfrac{1}{2}(\phi - \phi^{*_\mathbf{H}}), \qquad \Psi = \tfrac{1}{2}(\phi + \phi^{*_\mathbf{H}}).
$$
The first part of the statement follows from Lemma \ref{lemma:IJmmap}. As for the second part, we combine \eqref{eq:harmonicvsmmap} with
$$
(\nabla^{\mathbf{H}} J \Psi)  \wedge \omega^{n-1} = \tfrac{i}{2}(- \dbar_{\mathcal{Q}} \phi + \partial^{\mathbf{H}} \phi^{*_\mathbf{H}} + [\phi \wedge \phi^{*_\mathbf{H}}]) \wedge \omega^{n-1}.
$$
\end{proof}

We give next the proof of Proposition \ref{prop:harmHKstab}.

\begin{proof}[Proof of Proposition \ref{prop:harmHKstab}]
Let $\mathbf{H}$ be a harmonic metric for $D$. Let $\mathcal{F}\subset \mathcal{Q}$ be an isotropic subsheaf preserved by $D$. By \cite[Ch. V, Proposition 7.6]{Kobayashi}, there exists a reflexive subsheaf $\mathcal{F}_1\subset \mathcal{Q}$ such that $\mathcal{F} \subset \mathcal{F}_1$, $\mathcal{F}_1/\mathcal{F}$ is a torsion sheaf, and
$$
\mu_{\mathfrak{b}_0}(\mathcal{F}) \leqslant \mu_{\mathfrak{b}_0}(\mathcal{F}_1).
$$
Since the singular (analytic) set $S\subset X$ of $\mathcal{F}$ has $\mathrm{codim}\hspace{1mm} S\geq 2$, it follows by a density argument that $\mathcal{F}_1$ is also preserved by $D$. Hence, it suffices to assume that $\mathcal{F}$ is reflexive. 

In that case, there exists an analytic set $S\subset X$ of $\mathrm{codim}\hspace{1mm} S\geq 3$ and a holomorphic vector bundle $F$ defined on $X\backslash S$ such that $\mathcal{F}|_{X\backslash S}\cong \mathcal{O}(F)$. Denote by $E=\mathcal{Q}|_{X\backslash S}/F$. Using $\mathbf{H}$ we can make a smooth identification of $E$ and $F^{\perp_{\mathbf{H}}}$ on $X\backslash S$. In the splitting $\mathcal{Q}|_{X\backslash S} = F\oplus F^{\perp_{\mathbf{H}}}$ we have
$$
D^{\mathbf{H}}|_{X\backslash S}=\left(\begin{array}{c c}
D_{F}^{\mathbf{H}} & \beta\\
-\beta^{*_{\mathbf{H}}} & D_{E}^{\mathbf{H}}\end{array}\right)
$$
for $D_{F}^{\mathbf{H}}$, $D_{E}^{\mathbf{H}}$ the restricted Chern connections of $F$ and $E$ and some $\beta\in \Omega^{0,1}(\mathrm{Hom}(E,F))$. Similarly, for the (restricted) Higgs field we have
$$
\phi=\left(\begin{array}{c c} 
\phi_F & \theta\\
\beta^{*_{\mathbf{H}}} & \phi_{E}\end{array}\right) \hspace{2mm} , \hspace{2mm} \phi^{*_{\mathbf{H}}}=\left(\begin{array}{c c} 
\phi_F^{*_{\mathbf{H}}} & \beta\\
\theta^{*_{\mathbf{H}}} & \phi_{E}^{*_{\mathbf{H}}}\end{array}\right)
$$
since we assumed that $\mathcal{F}$ is $D$-preserved. Now, by Lemma \ref{lemma:IJKmmapHiggs},
$$
\left(F_{\mathbf{H}}+\tfrac{1}{2}[\phi\wedge \phi^{*_{\mathbf{H}}}]\right)\wedge \omega^{n-1}=0.
$$ 
and therefore
$$
0 = \big(F_{\mathbf{H}}+\tfrac{1}{2}[\phi\wedge \phi^{*_{\mathbf{H}}}]\big)|_F \wedge \omega^{n-1} = (F_{D_F^{\mathbf{H}}}-\tfrac{1}{2}\beta\wedge \beta^{*_{\mathbf{H}}}+\tfrac{1}{2}[\phi_F\wedge \phi_F^{*_{\mathbf{H}}}]+\tfrac{1}{2}\theta\wedge \theta^{*_{\mathbf{H}}}) \wedge \omega^{n-1}.
$$
Note that $\mathrm{tr} \hspace{0.5mm} F_{D^{\mathbf{H}}_F}$ is the restriction to $X \backslash S$ of a (smooth) representative of $-2\pi i c_1(\mathrm{det} \hspace{0.5mm} \mathcal{F})$). Then, taking traces in the previous expression and integrating over $X$ we get
$$
c_1(\mathcal{F})\cdot \mathfrak{b} +\tfrac{(n-1)!}{2\pi}\int_{X} \tfrac{1}{2}\mathrm{tr}_{F}(\beta\wedge (* \beta)^{*_{\mathbf{H}}})+\tfrac{1}{2} \mathrm{tr}_{F}(\theta\wedge (*\theta)^{*_{\mathbf{H}}}) = 0,
$$
where we used that $\mathrm{tr}_{F}[\phi_F\wedge \phi_F^{*_{\mathbf{H}}}]=0$. Note that the integral in the previous expression is nonnegative, since for any $\varphi\in \mathrm{Hom}(E,F)$
\begin{align*}
\mathrm{tr}_{F}(\varphi\circ \varphi^{*_{\mathbf{H}}})= &\mathrm{tr}_{F\oplus E}\left(\left(\begin{array}{c c}
0 & \varphi\\
0 & 0\end{array}\right)\left(\begin{array}{c c}
0 & 0\\
\varphi^{*_{\mathbf{H}}} & 0\end{array}\right)\right) \geq 0.
\end{align*}
Thus, it follows that $\mu_{\mathfrak{b}}(\mathcal{F})\leq \mu_{\mathfrak{b}}(\mathcal{Q})=0$ and hence $\mathcal{Q}$ is semistable. 

In case of equality, one has $\beta, \theta =0$ and we get a holomorphic splitting $\mathcal{Q}|_{X\backslash S}=\mathcal{F}|_{X\backslash S}\oplus \mathcal{Q}/\mathcal{F}|_{X\backslash S}$ which is furthermore preserved by $D$. Then, since $\mathcal{F}$ is reflexive, so are $\Hom(\mathcal{Q},\mathcal{F})$ and $\Hom(\mathcal{F},\mathcal{F})$; in particular they are normal. Then we can extend uniquely the projection map $r: \mathcal{Q}|_{X\backslash S}\rightarrow \mathcal{F}|_{X\backslash S}$ to $X$. Moreover, the composition with $j: \mathcal{F}\rightarrow \mathcal{Q}$ is the unique extension to $X$ of $\mathrm{Id}_{\mathcal{F}|_{X\backslash S}}$ and hence $r$ is a retraction for the exact sequence
$$
0\rightarrow \mathcal{F}\rightarrow \mathcal{Q}\rightarrow \mathcal{Q}/\mathcal{F}\rightarrow 0.
$$
We conclude that the sequence is split. Then, $\mathcal{F}$ and $\mathcal{Q}/\mathcal{F}$ are locally free and $D$-preserved, and the proof follows.

\end{proof}

We are ready to prove the main result of this section.

\begin{theorem}\label{thm:HSstab}
Let $(g,h_0,h_1)$ be a solution of the Hull-Strominger system \eqref{eq:HSabstract} with balanced class $\mathfrak{b}:= [\|\Omega\|_\omega \omega^{n-1}]$. Consider the associated triple $(\mathcal{Q},\IP{,},D^\mathbf{G})$ as in Proposition \ref{prop:HS-HE} and assume that $(\mathcal{Q},\IP{,})$ admits a harmonic metric $\mathbf{H}$. Then, $(\mathcal{Q},\IP{,},D^{\mathbf{G}})$ is slope $\mathfrak{b}$-polystable.
\end{theorem}

\begin{proof}
Observe that the $\mathbf{H}$ is also harmonic for $g' = \|\Omega\|_\omega^{\tfrac{1}{n-1}} g$ (see Lemma \ref{lemma:IJmmapApp}), which is balanced $d \omega'^{n-1} = 0$. Consequently, the result follows as a direct consequence of Proposition \ref{prop:HS-HE} and Proposition \ref{prop:harmHKstab}.
\end{proof}

\subsection{Non-holomorphic Higgs fields}\label{sec:higgsnon}

We establish next a comparison between the equations in Lemma \ref{lemma:IJKmmapHiggs} and the \emph{Hitchin's Equations} in the theory of Higgs bundles \cite{Hitchin1987}. The main qualitative difference between these quations is that the Higgs field $\phi$ in our picture is very often not holomorphic, as we can see from the following result.

\begin{lemma}\label{lem:nonholHiggs}

Let $X$ be a compact K\"ahler manifold endowed with a holomorphic volume form $\Omega$. Let $V_0$ and $V_1$ be holomorphic vector bundles over $X$ satisfying \eqref{eq:c1c2abstract}. Let $(g,h_0,h_1)$ be a solution of the Hull-Strominger system \eqref{eq:HSabstract} and consider the associated triple $(\mathcal{Q},\IP{,},D^\mathbf{G})$ as in Proposition \ref{prop:HS-HE}. Let $\mathbf{H}$ be compatible Hermitian metric on $(\mathcal{Q},\IP{,})$ such that the associated Higgs field in \eqref{eq:C-decomposition} satisfies
$$
\dbar_{\mathcal{Q}} \phi \wedge \omega^{n-1}  = 0.
$$
Assume that $\mathfrak{b}:= [\|\Omega\|_\omega \omega^{n-1}]$ is a $(n-1)^{\mathrm{th}}$-power of a K\"ahler class. Then, $g$ is K\"ahler and $h_0$ and $h_1$ are flat.
\end{lemma}
\begin{proof}
By \eqref{eq:C-decomposition}, we have that
$$
F_{\mathbf{G}} \wedge \omega^{n-1} = (F_{\mathbf{H}} + \dbar_{\mathcal{Q}} \phi) \wedge \omega^{n-1}=0.
$$
By Theorem \ref{t:DUY}, $\dbar_{\mathcal{Q}} \phi \wedge \omega^{n-1}  = 0$ implies that $(\mathcal{Q},\IP{,})$ is slope $\mathfrak{b}$-polystable and hence the statement follows from Proposition \ref{prop:HSnogo}.
\end{proof}

\begin{remark}\label{rem:connectionvsbracket}
In order to relate our stability condition to a practical obstruction to the existence of solutions to \eqref{eq:HSabstract}, it seems necessary to establish a more clear relation between the Dorfman bracket $[,]$ on $\mathcal{Q}$ (see Remark \ref{rem:bracket}) and the orthogonal connection $D^\mathbf{G}$, in a way that the slope inequality is formulated more naturally in terms of the triple $(\mathcal{Q},\IP{,},[,])$. 
\end{remark}

To finish this section, we provide an explicit formula for the Higgs field of the Hermitian metric given in \eqref{eq:H}. We will apply this formula to provide a non-K\"ahler example of non-holomorphic Higgs field associated to a harmonic metric in Section \ref{subsec:HiggsIwa}. The proof follows from tedious but straightforward calculations, and is ommited. 

\begin{lemma}\label{l:exphiggs}
Let $(X,\Omega,V_0,V_1)$ be as in the statement of Proposition \ref{prop:HS-HE}. Let $(g,h_0,h_1)$ be a solution of the Hull-Strominger system \eqref{eq:HSabstract} with $\alpha > 0$ and consider the associated triple $(\mathcal{Q},\IP{,},D^\mathbf{G})$ and the (positive) hermitian metric $\mathbf{H}$ defined in \eqref{eq:H}. Then, the Higgs field $\phi$ corresponding to  $\mathbf{H}$ via the decomposition \eqref{eq:C-decomposition}
is given by
$$
\phi=\left(\begin{array}{c c c}
0 & 2\alpha g^{-1}\mathrm{tr}_{V_0}(\mathbb{F}^{1,0}_{h_0}\cdot) & 0\\
-2\mathbb{F}_{h_0}^{1,0} & 0 & 0\\
0 & 0 & 0
\end{array}\right),
$$
where we are using the identification $\mathcal{Q}=T_\mathbb{C}\oplus \mathrm{End}\ V_0\oplus \mathrm{End}\ V_1$, and where 
\begin{align*}
i_{V^{1,0}}\mathbb{F}^{1,0}_{h_0}(W)=F_{h_0}(V^{1,0},W),\\
i_{V^{1,0}}g^{-1}\mathrm{tr}_{V_0}(\mathbb{F}^{1,0}_{h_0}r_0)=g^{-1}(\mathrm{tr}_{V_0}(i_{V^{1,0}}F_{h_0} r_0)).
\end{align*}
Moreover
\begin{align*}
\tiny
i_{W^{0,1}}i_{V^{1,0}}(\overline{\partial}_\mathcal{Q}\phi)\left(\begin{array}{c}
X\\
r_0\\
r_1
\end{array}\right)
\normalsize
=  \tiny\left(\begin{array}{c}
\begin{array}{c}
2\alpha g^{-1}\big(\mathrm{tr}_{V_0}(\overline{\partial}^{\nabla^-,V_0}\mathbb{F}_{h_0}^{1,0}(V^{1,0},W^{0,1})\cdot r_0\\
\hspace{6mm}+i_{V^{1,0}}F_{h_0}F_{h_0}(W^{0,1},X))-i_{W^{0,1}}F_{h_0}F_{h_0}(V^{1,0},X))\big)
\end{array}\\
\\
\begin{array}{c}
i_{W^{0,1}}i_{V^{1,0}}(-2\overline{\partial}^{\nabla^-,V_0}\mathbb{F}_{h_0}^{1,0}(X^{0,1})-2\alpha i_{e_i}F_{h_0}\wedge \mathrm{tr}_{V_0}(i_{e_i}F_{h_0}\cdot r_0))\\
\hspace{-31mm}+2\alpha F_{h_0}(e_i,V^{1,0}) \mathrm{tr}_{V_1}(F_{h_1}(e_i,W^{0,1})\cdot r_1))
\end{array}\\
\\
2\alpha F_{h_1}(e_i,W^{0,1})\mathrm{tr}_{V_0}(F_{h_0}(e_i,V^{1,0})r_0)\normalsize
\end{array}\right)
\end{align*}
where $\{e_i\}$ is a $g$-orthonormal frame, and $\overline{\partial}^{\nabla^-,V_0}$ stands for the Dolbeault operator induced in $\Omega^{1,0}(\mathrm{Hom}(T^{0,1},\mathrm{End}\ E_0))$ by $(\nabla^-_{\mathbb{C}})^{0,1}$ and $\overline{\partial}^{V_0}$.

\end{lemma}


\section{Examples}\label{sec:examples}

\subsection{A family of solutions on the Iwasawa manifold}\label{subsec:Iwasawa}
In Section \ref{sec:higgsstb} we have proved that triples $(\mathcal{Q},\IP{,},D^\mathbf{G})$ associated to solutions of the Hull-Strominger system are polystable in the sense of Definition \ref{d:HKstab}, provided that they admit a harmonic metric (see Definition \ref{def:harmonic}). The aim of this section is to give some examples where one has a positive answer to Question \ref{questionharmonic}.

Consider the \textit{Iwasawa manifold} $X =\Gamma \backslash H_\CC$, given by the quotient of the complex Heisenberg Lie group
$$
H_\CC =\left\{ \left(\begin{array}{c c c}
1 & z_2 & z_3\\
0 & 1 & z_1\\
0 & 0 & 1\\
\end{array}\right) \hspace{1mm} |\hspace{1mm} z_i \in \mathbb{C}\right\}
$$
by the lattice $\Gamma \subset \CC$ of Gaussian matrices (with entries in $\mathbb{Z}[i]$). The manifold $X$ has a natural complex structure $J$, induced by the bi-invariant complex structure on $H_\CC$. Note that there is a holomorphic projection to the standard complex torus
$$
p: X \rightarrow T^4=\mathbb{C}^2/\mathbb{Z}[i]^2
$$
which makes $X$ a holomorphic torus fibration. Let 
$$
\omega_1=dz_1 \hspace{2mm} , \hspace{2mm} \omega_2=dz_2 \hspace{2mm} , \hspace{2mm} \omega_3=dz_3-z_2 dz_1
$$
be $1$-forms on the Lie group $H_\CC$. These forms are lattice invariant and descend to $X$. Moreover, they define a global basis of $T^*_{1,0}$ and satisfy the structure equations
$$d\omega_1=d\omega_2=0 \hspace{3mm}, \hspace{3mm} d\omega_3=\omega_{12}$$
from which all the exterior algebra relations can be derived. The inclusion of the invariant subcomplex
induces isomorphism on all de Rham and the complex Dolbeault, Bott-Chern and Aeppi cohomology groups, (see e.g. \cite{An}).  

For any choice of
$$
(m,n,p)\in \mathbb{Z}^{3}\backslash \{0\}
$$
we consider the following purely imaginary $(1,1)$-form on the base $T^4$
\begin{equation}\label{lbtorus}
F=\pi (m(\omega_{1\overline{1}}-\omega_{2\overline{2}})+n(\omega_{1\overline{2}}+\omega_{2\overline{1}})+ip(\omega_{1\overline{2}}-\omega_{2\overline{1}})).
\end{equation}
Note that $\tfrac{i}{2\pi} F$ has integral periods and hence, by general theory, this is the curvature form of the Chern connection of a holomorphic hermitian line bundle $(\mathcal{L},h)\rightarrow T^4$. In the sequel, we will identify $(\mathcal{L},h)$ and $F = F_h$ with their corresponding pull-backs to $X$ via $p$. 

Fix $(m_0,n_0,p_0), (m_1,n_1,p_1) \in \mathbb{Z}^3 \backslash \{0\}$ and consider the associated holomorphic Hermitian bundles $(\mathcal{L}_j,h_j) \to X$, for $j = 0,1$. Consider the $SU(3)$ structure on $X$ defined by 
\begin{align}\label{Iwasu3}
    \Omega=\omega_{123} \hspace{3mm}, \hspace{3mm} \omega_0=\tfrac{i}{2}(\omega_{1\overline{1}}+\omega_{2\overline{2}}+\omega_{3\overline{3}}). 
\end{align}
Note that $\omega_0$ is a balanced hermitian metric and $\Omega$ is a holomorphic volume form.

\begin{proposition}\label{prop:HSsol}
With the notation above, the triple $(g_0,h_0,h_1)$ is a solution of the Hull-Strominger system \eqref{eq:HSabstract} on $(X,\Omega,\cL_0,\cL_1)$ if and only if
\begin{equation}\label{eq:alphasol}
\alpha=\frac{1}{2\pi^2(m_0^2+n_0^2+p_0^2-m_1^2-n_1^2-p_1^2)}. 
\end{equation}
\end{proposition}

\begin{proof}
The first two equations of the system follow from the fact that $\cL_0$ and $\cL_1$ have degree zero combined with the fact that $F_{h_0}$ and $F_{h_1}$ are induced by left-invariant forms. The conformally balanced equation follows from $d\omega_0^2 = 0$ and the fact that $\|\Omega\|_{\omega_0}$ is constant. Finally,
\begin{equation*}
\begin{split}
dd^c \omega_0 & = \omega_{12\overline{12}} ,\\
F_{h_0}^2 & =  2\pi^2(m_1^2+n_1^2+p_1^2)\omega_{12\overline{12}},\\
F_{h_1}^2 & = 2\pi^2(m_2^2+n_2^2+p_2^2)\omega_{12\overline{12}}.
\end{split}    
\end{equation*}
and hence the Bianchi identity, given by the last equation in \eqref{eq:HSabstract}, is equivalent to \eqref{eq:alphasol}.
\end{proof}

\subsection{Existence of harmonic metrics}\label{subsec:exhar}

Recall that there is a fibration structure $p \colon X \to T^4$ and that $\cL_i$ are pull-back from the base.  Let $\underline{P}$ be the smooth complex principal bundle underlying the bundle of split frames of $\mathcal{L}_0 \oplus \mathcal{L}_1$. 
The set of holomorphic structures on $\underline{P}$ is a torsor for 
$$
B=\mathrm{Pic}^0_X\times \mathrm{Pic}^0_X,
$$
which is a complex $4$-dimensional Abelian variety, since $H^{0,1}_{\dbar}(X) \cong \CC^2$ (see \cite{An}). We will denote by $P_x$ the holomorphic bundle associated to $x\in B$. More precisely, we will identify $P_x$ with the holomorphic bundle of split frames of a direct sum of line bundles $\mathcal{L}_0^x \oplus \mathcal{L}_1^x$, where $\mathcal{L}_0^0 \oplus \mathcal{L}_1^0 = \mathcal{L}_0 \oplus \mathcal{L}_1$.

To state the main result of this section, for each $x \in B$ we need to consider holomorphic orthogonal bundles $(\mathcal{Q}_x,\IP{,})$ such that $\mathcal{Q}_x$ is a holomorphic extension of the holomorphic Atiyah algebroid $A_{P_x}$ of $P_x$ by the holomorphic cotangent bundle (cf. Lemma \ref{lem:HSQ})
\begin{equation}\label{eq:holCoustrx}
0 \longrightarrow T^*_{1,0} \overset{\pi^*}{\longrightarrow} \mathcal{Q}_x \longrightarrow A_{P_x} \longrightarrow 0.
\end{equation}
We are interested in extensions of this form which may arise from solutions of the Hull-Strominger system (see Remark \ref{rem:BC}). By \cite[Remark 2.13]{GFGM}, these are parametrized by the image of the natural map
\begin{equation}\label{eq:partialmap}
\partial \colon H^{1,1}_A(X,\mathbb{R}) \to H^{2,1}_{\dbar}(X),
\end{equation}
where $H^{p,q}_{A}(X)$ are the Aeppli cohomology groups of $X$, defined by
\begin{equation*}
H^{p,q}_{A}(X)=\frac{\mathrm{ker} \hspace{1mm} dd^{c}: \Omega^{p,q}(X,\mathbb{C}) \longrightarrow \Omega^{p+1,q+1}(X,\mathbb{C})}{\mathrm{Im} \hspace{1mm}  \partial \oplus \dbar: \Omega^{p-1,q}(X,\mathbb{C}) \oplus \Omega^{p,q-1}(X,\mathbb{C}) \longrightarrow \Omega^{p,q}(X,\mathbb{C})}
\end{equation*}
and $H^{p,p}_{A}(X,\mathbb{R}) \subset H^{p,p}_{A}(X)$ is the canonical real structure. Observe here that the construction of the Dolbeault operator in Lemma \ref{lem:HSQ} can be modified in the following way: given $[\tau] \in H^{1,1}_A(X,\mathbb{R})$ we can change $2i\partial \omega \to 2i\partial (\omega + \tau)$, which still defines an integrable Dolbeault operator. Observe furthermore that this induces a new holomorphic orthogonal bundle structure on \eqref{eq:Qexpb} with the same pairing.

\begin{proposition}\label{exhar}
Let $(m_i,n_i,p_i)\in \mathbb{Z}^3\backslash\{0\}$, $i=0,1$ such that
\begin{equation}\label{bilinearF}
c_1(\cL_0)\cdot c_1(\cL_1)=0\in H^4_{dR}(T^4,\mathbb{R}).
\end{equation}
We fix the coupling constant $\alpha$ as in \eqref{eq:alphasol}. Then, for any $x\in B$ there exists a holomorphic orthogonal bundle $(\mathcal{Q}_x^0,\IP{,})$ induced by a solution of the Hull-Strominger system on $(X,\Omega,\cL_0^x,\cL_1^x)$ which admits a harmonic metric. Furthermore, for any small deformation $(\mathcal{Q}_x,\IP{,})$ of $(\mathcal{Q}_x^0,\IP{,})$ parametrized by an element in the image of \eqref{eq:partialmap}, there exists a solution to the Hull-Strominger system inducing $(\mathcal{Q}_x,\IP{,})$ and a harmonic metric for this solution.
\end{proposition}
\begin{proof}
We proceed in steps. Firstly, we prove the result for a single holomorphic orthogonal bundle as in the statement.
Then, we check that this construction is stable with respect to deformations. The holomorphic pairing will remain constant along the family.

\textbf{Step 1.} Let $x = (x_0,x_1) \in B$ and choose lifts $\tilde x_i \in H^{0,1}_{\overline{\partial}}(X)$. We denote by $\cL_i'$ the holomorphic line bundle corresponding to $\tilde x_i$. The Chern connections of $h_i$ on $\mathcal{L}_i$ and $\mathcal{L}_i'$ are related by
$$
\theta'_i =\theta_i + a_i, \; i=0,1
$$
where $a_i^{0,1}\in \Omega^{0,1}_{\mathrm{inv}}$ is an invariant form representative of $\tilde x_i$. Observe that
$$
F_{h_i'}=F_{h_i}+da_i=F_{h_i}
$$
since $F_{h_i'}, F_{h_i}\in \Omega^{1,1}$ and $da_i\in \Omega_{\mathrm{inv}}^{2,0}\oplus \Omega_{\mathrm{inv}}^{0,2}$ for $a_i$ invariant. From this, $(\omega_0,\theta_0',\theta_1')$ is a solution of the Bianchi identity, and we consider the associated Bott-Chern algebroid $\mathcal{Q}^0_x$. Then, $\mathcal{Q}^0_x$ admits a solution to the Hull-Strominger system $(g_0,h_0,h_1)$. Furthermore, we obtain a harmonic metric on $\mathcal{Q}^0_x$
$$
\mathbf{H}_0=\left(\begin{array}{c c c}
g_0 & 0 & 0\\
0 & -\alpha & 0\\
0 & 0 & -\alpha
\end{array}\right)
$$
under the topological constraints \eqref{bilinearF} in the statement. To see this, by Lemma \ref{lemma:Kmmap}, we must have
$$
F_{h_0}\wedge * d^c\omega_0=0 \hspace{2mm} , \hspace{2mm} \sum_{i,j} F_{h_0}(e^0_i,e^0_j)F_{h_1}(e^0_i,e^0_j) =0.
$$
for any $\omega_0$-orthonormal basis $\{e_i^0\}$. Using the expression \eqref{lbtorus} for the curvature $F_{h_0}$ combined with 
$$
* d^c\omega_0 = \tfrac{i}{2}(\omega_{12\overline{3}}-\omega_{3\overline{12}})
$$
we get that the first of these equations holds for any value of the integers $(m_0,n_0,p_0)\in \mathbb{Z}^3\backslash \{0\}$. Using that $F_{h_i}$ are Hermitian-Yang-Mills, the second equation may be rewritten as
$$
F_{h_0}\wedge F_{h_1}\wedge \omega_0=0
$$
or, in terms of the parameters,
$$
m_0m_1+n_0n_1+p_0p_1=0.
$$
Finally, using that $F_{h_i}$ are pull-back from the base torus $T^4$, one can easily see that this condition is equivalent to \eqref{bilinearF}.

\textbf{Step 2.} It is easy to check that a basis of $H^{1,1}_A(X,\RR)/\mathrm{ker}\hspace{0.5mm} \partial$ is given by the classes of the real $(1,1)$-forms
$$
\tau_1=\omega_{1\overline{3}}-\omega_{3\overline{1}},\; \tau_2=i(\omega_{1\overline{3}}+\omega_{3\overline{1}}),\; \tau_3=\omega_{2\overline{3}}-\omega_{3\overline{2}},\; \tau_4=i(\omega_{2\overline{3}}+\omega_{3\overline{2}}).
$$
Then, any holomorphic orthogonal bundle in the deformation family of $(\mathcal{Q}^0_x,\IP{,})$ (as in the statement) is isomorphic to $(\mathcal{Q}_x^\tau,\IP{,})$ where $\tau=\sum_{i=1}^4 t_i\tau_i$, $t_i\in \mathbb{R}$, $\mathcal{Q}_x^\tau$ is the holomorphic vector bundle with Dolbeault operator \eqref{eq:DolQ} with $2i\partial \omega$ replaced by $2i\partial (\omega + \tau)$, and $\IP{,}$ is constant given by \eqref{eq:pairing}. Now, one can readily check that, for any small $\tau$, $(g_0 + \tau(\cdot,J\cdot),h_0,h_1)$ is a solution to the Hull-Strominger system which induces $(\mathcal{Q}_x^\tau,\IP{,})$. Furthermore, applying Lemma \ref{lemma:Kmmap} the harmonicity conditions for the metric
\begin{equation}\label{eq:exharm}
\mathbf{H}_\tau =\left(\begin{array}{c c c}
g_0 + \tau(\cdot,J\cdot) & 0 & 0\\
0 & -\alpha & 0\\
0 & 0 & -\alpha
\end{array}\right) 
\end{equation}
follow again from a straightforward calculation which implies
$$
F_{h_0}\wedge *_{\omega_0 + \tau} d^c(\omega_0 + \tau)=0,
$$
combined with condition \eqref{bilinearF}.
\end{proof}

\begin{remark}\label{rem:BC}
More invariantly, in the formalism of Bott-Chern algebroids introduced in \cite{GaRuShTi}, the previous result can be stated as follows: for any $x\in B$ there exists a Bott-Chern algebroid with underlying bundle $P_x$ and fixed Lie algebra bundle determined by the pairing
$$
\IP{(r_0,r_1),(r_0',r_1')} : = - \alpha r_0 r_0' + \alpha r_1 r_1',
$$
such that any small Bott-Chern algebroid deformation admits a solution to the Hull-Strominger system and a harmonic metric for this solution.  
\end{remark}

As an immediate consequence of the previous result and Theorem \ref{thm:HSstab}, we obtain families of examples of holomorphic orthogonal bundles with connection associated to solutions of the Hull-Strominger system which satisfy the stability condition in Definition \ref{d:HKstab}.

\begin{corollary}
The orthogonal vector bundles with connection 
$$
(\mathcal{Q}_x,\langle \cdot,\cdot\rangle, D^{\mathbf{G}})
$$
given by Proposition \ref{exhar} are $\mathfrak{b}$-polystable in the sense of Definition \ref{d:HKstab}, where $\mathfrak{b}$ is the balanced class of the solution which induces $(\mathcal{Q}_x,\langle \cdot,\cdot\rangle)$.
\end{corollary}
\begin{proof}
This is immediate from Proposition \ref{prop:harmHKstab} and Proposition \ref{exhar}. 
\end{proof}

\subsection{Higgs fields on the Iwasawa manifold}\label{subsec:HiggsIwa}

As we discussed in Section \ref{sec:higgsnon}, we cannot expect the Higgs field defined by a harmonic metric $\mathbf{H}$ for the Hull-Strominger system to be holomorphic. Motivated by Lemma \ref{lem:nonholHiggs}, the aim of this section is to provide a non-K\"ahler example which illustrates this phenomenon.

We shall focus on the family of examples constructed in Proposition \ref{exhar}, with Harmonic metrics $\mathbf{H}_\tau$ defined in \eqref{eq:exharm}. The proof follows by application of Lemma \ref{l:exphiggs}.

\begin{proposition}
Let $(g_0+\tau(\cdot,J\cdot),h_0,h_1)$ be the solution of the Hull-Strominger system in $(X,\Omega,\mathcal{L}^x_0,\mathcal{L}^x_1)$ constructed in Proposition \ref{exhar}, for $\alpha$ given by \eqref{eq:alphasol}. Let $\mathbf{H}_\tau$ be the harmonic metric for this solution constructed in the proof of Proposition \ref{exhar}. Then, for any small $\tau$, the Higgs field $\phi$ of $\mathbf{H}_\tau$ satisfies
$$
\overline{\partial}_\mathcal{Q}\phi \neq 0.
$$
\end{proposition}
\begin{proof}
The proof follows by an explicit calculation using Lemma \ref{l:exphiggs}. Notice that, in our examples, both $V_0$ and $V_1$ are line bundles, and therefore $\mathrm{End}\ V_i \cong \mathbb{C}$ canonically. Applying Lemma \ref{l:exphiggs}, we can write $\overline{\partial}_\mathcal{Q}\phi$ in matrix form corresponding to the splitting $\mathcal{Q}=T_\mathbb{C}\oplus \mathbb{C} \oplus \mathbb{C}$, as follows
\begin{align*}
\overline{\partial}_\mathcal{Q}\phi = \left(\begin{array}{c c c}
(\overline{\partial}_\mathcal{Q}\phi)_{11} & (\overline{\partial}_\mathcal{Q}\phi)_{12} & 0\\
(\overline{\partial}_\mathcal{Q}\phi)_{21} & (\overline{\partial}_\mathcal{Q}\phi)_{22} & (\overline{\partial}_\mathcal{Q}\phi)_{23}\\
0 & (\overline{\partial}_\mathcal{Q}\phi)_{32} & 0
\end{array}\right).
\end{align*}
whose components are 
 $(1,1)$-forms with values in $\mathrm{End} \ T_\mathbb{C}$ for $(\overline{\partial}_\mathcal{Q}\phi)_{11}$, in $T_\mathbb{C}$ for $(\overline{\partial}_\mathcal{Q}\phi)_{12}$, in $T^*_\mathbb{C}$ for
 $(\overline{\partial}_\mathcal{Q}\phi)_{21}$, and are scalar for the rest.

By continuity on the parameter $\tau$, it is enough to check that $\overline{\partial}_\mathcal{Q}\phi\neq 0$ for $\tau=0$. It follows from a straightforward computation that:
\begin{align*}
(\overline{\partial}_\mathcal{Q}\phi)_{23}= -4\pi^2\alpha(&(m_0m_1+(n_0+ip_0)(n_1-ip_1))\omega_{1\overline{1}}+\\
+&(m_0(n_1+ip_1)-m_1(n_0+ip_0))\omega_{1\overline{2}}+\\
+&(-m_0(n_1-ip_1)+m_1(n_0-ip_0))\omega_{2\overline{1}}+\\
+&(m_0m_1+(n_0-ip_0)(n_1+ip_1))\omega_{2\overline{2}}).
\end{align*}
Finally, subject to the condition (\ref{bilinearF}) together with $(m_i,n_i,p_i)\in \mathbb{Z}^3\backslash\{0\}$, it is easy to check that $(\overline{\partial}_\mathcal{Q}\phi)_{23}\neq 0$.

\end{proof}

\end{document}